\documentclass{article}
\usepackage{graphicx} 
\usepackage{subfig}
\usepackage{amsfonts}
\usepackage{amsmath}
\usepackage{amssymb}
\usepackage{url}
\usepackage{fancyhdr}
\usepackage{indentfirst}
\usepackage{enumerate}
\usepackage{longtable}
\usepackage{float} 
\usepackage{pifont}
\usepackage{tikz}
\usepackage{tikz-cd}
\usepackage{etoolbox}
\newcommand{\circled}[2][]{\tikz[baseline=(char.base)]
    {\node[shape = circle, draw, inner sep = 1pt]
    (char) {\phantom{\ifblank{#1}{#2}{#1}}};%
    \node at (char.center) {\makebox[0pt][c]{#2}};}}
\robustify{\circled}
\usepackage{bbm}

\usepackage{dsfont}
\usepackage[colorlinks=true,citecolor=blue]{hyperref}
\usepackage{amsthm}
\usepackage{color}
\usepackage{natbib}

\usepackage{comment}
\usepackage{capt-of}
\usepackage{multirow}
\numberwithin{equation}{section}
\def\d{\mathrm{d}}

\newcommand{\E}{\mathbb{E}}

\newcommand{\F}{\mathcal{F}}

\newcommand{\R}{\mathbb{R}}
\newcommand{\N}{\mathbb{N}}
\newcommand{\p}{\mathbb{P}}

\newcommand{\M}{\mathcal{M}}

\renewcommand{\[}{\left[}

\renewcommand{\ge}{\geqslant}
\renewcommand{\le}{\leqslant}
\renewcommand{\geq}{\geqslant}
\renewcommand{\leq}{\leqslant}
\renewcommand{\epsilon}{\varepsilon}

\theoremstyle{plain}
\newtheorem{theorem}{Theorem}

\newtheorem{proposition}{Proposition}
\theoremstyle{definition}
\newtheorem{definition}{Definition}
\newtheorem{example}{Example}

\newtheorem{assumption}{Assumption}
\theoremstyle{remark}
\newtheorem{remark}{Remark}

\newcommand{\cet}{\begin{center}}
	\newcommand{\ecet}{\end{center}}

\usepackage{refcount}
\setrefcountdefault{1}

\newcounter{saveexample}

\allowdisplaybreaks[4]
\usepackage{setspace}

\title{
Time-Inconsistent Singular Control Problems with a Running Minimum Process}

\author{
	 Rui Dai\thanks{\footnotesize Department of Mathematical Sciences, Tsinghua University, China. Email: \texttt{dair25@mails.tsinghua.edu.cn}}
     \and
     Guohui Guan\thanks{\footnotesize Center for Applied Statistics and School of Statistics, Renmin University of China, Beijing 100872, China. Email:\texttt{ guangh@ruc.edu.cn}}
    \and
    Zongxia Liang\thanks{\footnotesize Department of Mathematical Sciences, Tsinghua University, China. Email: \texttt{liangzongxia@tsinghua.edu.cn}}
	\and 	
    Xiaodong Luo\thanks{\footnotesize Department of Mathematical Sciences, Tsinghua University, China. Email: \texttt{luoxd21@mails.tsinghua.edu.cn}}
}

\begin{document}
\date{}

\maketitle

\begin{abstract}	

This paper develops a time-inconsistent and path-dependent singular control framework incorporating a running minimum process. We derive a verification theorem that characterizes equilibria under substantially weaker regularity conditions than those imposed in the existing literature, and we obtain a stronger notion of equilibrium by enlarging the class of feasible perturbations. We first establish the mathematical foundations of the framework by proving the existence and uniqueness of strong solutions to a class of Skorokhod reflection problems involving the running minimum and by characterizing admissible singular control laws. We further demonstrate the existence of an equilibrium through a dividend problem, where the running minimum leads to a highly coupled and nonlinear differential-algebraic equation system. For this problem, we prove the monotonicity and local concavity of the dividend boundary, thereby providing a mathematical explanation for dividend smoothing and scarring effects. Numerical simulations confirm the robustness of the equilibrium across a wide range of parameter values.

\vspace{0.1in}	
    

\end{abstract}
\section{Introduction}\label{sec:intr}
Singular control is typically modeled as an adapted, c\`{a}dl\`{a}g, nondecreasing stochastic process that represents the cumulative amount of control or intervention. The study of singular control problems originated in spacecraft control; see, e.g., \citet*{bather1967,Bene1980}. Since then, singular control has been extensively developed and applied to a wide range of problems in mathematical finance, including portfolio optimization with transaction costs, irreversible investment, and resource extraction; see, e.g., \citet*{davis1990,KL2001,ferrari2018}. Among these applications, the optimal dividend problem has been a particularly active research area and serves as the motivating context for the model studied in this paper. Representative studies on singular control approaches to optimal dividends include \citet*{jeanblanc1995,asmussen1997,kulenko2008}.

Time inconsistency, first systematically studied by \citet*{S1937}, arises frequently in financial and economic problems. Typical sources include non-exponential discounting, such as hyperbolic discounting, and nonlinear criteria, such as the mean-variance objective. Singular control problems with time inconsistency have recently been studied by \citet*{LLY2024}, \citet*{liang2025}, \citet*{dai2024}, and \citet*{CZ2025}. Specifically, \citet*{LLY2024} examine singular control under non-exponential discounting, whereas \citet*{liang2025}, \citet*{dai2024}, and \citet*{CZ2025} investigate related problems under the mean-variance criterion, albeit in different model settings. These studies also differ in their equilibrium formulations. In \citet*{LLY2024} and \citet*{liang2025}, the post-perturbation equilibrium control is governed by a reflection region determined through a Skorokhod reflection problem. By contrast, \citet*{dai2024} and \citet*{CZ2025} prescribe a specific form of the equilibrium control that is independent of the wealth process.

However, existing studies, including \citet*{LLY2024} and \citet*{CZ2025}, still have several limitations. First, they do not incorporate historical performance or other path-dependent effects. In these models, the payoff functional depends only on the current state, and the resulting equilibrium strategies are Markovian. Second, the absence of well-posedness for Skorokhod problems with coefficients depending on the control poses a intrinsic limitation on the admissible singular control laws within the framework of \cite{LLY2024}. Third, the verification theorems in \cite{LLY2024} impose relatively stringent regularity conditions, which may exclude admissible equilibrium candidates and, in some cases, lead to the non-existence of equilibrium solutions. 

To capture path-dependent effects, we introduce the running minimum process to represent historical performance. This modeling choice is well supported by the literature. In the context of scarring effects, \citet*{KJL2020} show that major shocks can inflict long-term damage on productivity, suggesting that adverse historical outcomes may have persistent effects on future performance. Moreover, as documented by \citet*{K2021}, although corporate profits recovered after the financial crisis, dividend growth lagged significantly, which is consistent with dividend smoothing theory and dividend signaling theory; see \citet*{Lintner1956}. In addition, \citet*{DO2008} provide evidence that firms with poorer historical performance tend to adopt more conservative payout policies. These findings justify incorporating a path-dependent process such as the running minimum.


In this context, literature such as \citet*{PG1998} and \citet*{GR2014} examine a continuous running maximum process but do not impose control on it. \citet*{FR2025} categorize integrals concerning singular controls and the running minimum into three types, constructing related payoff functionals to formally establish their connection. This work can be viewed as an extension of \citet*{ZH1992}, which first defined such integrals with respect to singular controls. Using these novel definitions, \citet*{DJ2025} study time-consistent optimization problems involving a tipping point, where the running minimum serves as a key auxiliary process.

Our paper adopts the framework of \citet*{LLY2024} and the novel definitions of integrals in \citet*{FR2025} to investigate a non-exponentially discounted, time-inconsistent singular control problem featuring a running minimum process. In contrast to studies on the running minimum (\citet*{FR2025} and \citet*{DJ2025}), our work is developed within a novel equilibrium framework. Compared with existing equilibrium frameworks (\citet*{LLY2024}, \citet*{dai2024}, and \citet*{CZ2025}), we overcome their limitations, such as excessively stringent regularity conditions.

We define the equilibrium control using an SDE and a Skorokhod reflection problem incorporating a running minimum process and a singular control. We establish the existence and uniqueness of a strong solution for the SDE and the reflection problem using the Banach fixed-point theorem and the iteration method; see Propositions \ref{prop:solution1} and \ref{prop:Skorokhod}. Based on these results, we construct the time-inconsistent singular control model in which the payoff functional and the reflection region are associated with the running minimum.

We then derive the verification theorem (Theorem \ref{th:verify}) using an extended change-of-variable formula. The running minimum process introduces a Neumann condition in the associated partial differential equation (abbr. PDE). In Proposition \ref{prop:f^s}, we provide concrete, verifiable sufficient conditions for the integrability requirements, thereby relaxing the $O(h)$-conditions from earlier literature. Theorem \ref{th:nec} further establishes necessary conditions for the extended HJB system, revealing that the variational inequality originates from the underlying second-order smooth condition.

We apply our verification theorem to a time-inconsistent dividend problem. We focus on time-independent equilibria and the payoff functional is shown to satisfy a system of free-boundary PDEs. By seeking separable solutions, the problem is reduced to a differential-algebraic equation (abbr. DAE), and further to a nonlinear ordinary differential equation (abbr. ODE) system. Under certain parameter conditions, we prove the existence and uniqueness of a solution to this ODE, thereby establishing the corresponding equilibrium (Proposition \ref{prop:spec}). Furthermore, the dividend barrier is proved to be monotonically decreasing and locally concave with respect to the historical wealth minimum (Proposition \ref{prop:b}). Numerical simulations confirm the prevalence of such equilibria. A comparative static analysis reveals how the dividend barrier varies with key parameters.

The main contributions of this paper are fourfold. First, to the best of our knowledge, this paper is among the first to incorporate the running minimum process into a time-inconsistent singular control framework. Employing an iterative method, we establish the existence and uniqueness of the strong solution to a class of Skorokhod reflection problems involving the running minimum and provide a sufficient characterization for admissible singular control laws.
Second, we propose a verification theorem characterizing the equilibrium under much weaker regularity conditions and more verifiable integrability conditions compared with the previous literature, such as \citet*{LLY2024} and \citet*{dai2024}. Moreover, expanding the set of feasible perturbations, we achieve a stronger notion of equilibrium. Third, we demonstrate the existence of equilibrium through a dividend problem, which is substantially more complex than models without the running minimum, where the equilibrium functional is solely determined by the solution of an ODE. The complexity stems mainly from the strong coupling of the algebraic‑differential equations and the high nonlinearity of the system. We prove the monotonicity and local concavity of the dividend boundary. These findings align closely with dividend smoothing theory, dividend signaling theory, and the scarring effect in economic phenomena that previous models could not adequately explain.
Fourth, numerical simulations and comparative static analyses are carried out for this dividend problem, confirming the prevalence of equilibria across a wide range of parameter configurations. Combining path‑dependent dynamics with a subgame‑perfect Nash equilibrium framework, our model captures the impacts of both historical trajectories and future strategic interactions.

The remaining part of our paper is organized as follows:  Section \ref{sec:model} introduces a time-inconsistent singular control framework that incorporates a running minimum process and provides sufficient conditions for admissible controls. Section \ref{sec:fin} presents a verification theorem under weaker regularity conditions and establishes necessary conditions for the extended HJB system. The existence of an equilibrium is then demonstrated through a concrete example in Subsection \ref{subsec:rigor}. Finally, numerical simulations and a comparative static analysis in Subsection \ref{subsec:num} illustrate the equilibrium and its economic implications.
\section{Model Setting}\label{sec:model}
Let $(\Omega,\mathcal{F},\mathbb{P})$ be a complete probability space equipped with a one-dimensional standard Brownian motion $\{B_t\}_{t\geq 0}$. The filtration $\mathbb{F}:=\{\mathcal{F}_t\}_{t\geq 0}$ is the $\mathbb{P}$-augmentation of the natural filtration generated by $\{B_t\}_{t\geq 0}$ and satisfies the usual conditions.

The state process $\{X_t^{\mathrm{D}}\}_{t\geq0}$ and the running minimum process $\{M_t^{\mathrm{D}}\}_{t\geq0}$ evolve following the general stochastic differential equation,\begin{small}
\begin{align}\label{eq:X0}
\begin{cases}
\d X_{t}^{\mathrm{D}} = \mu(X_{t}^{\mathrm{D}},M_{t}^{\mathrm{D}},\mathrm{D}_{t},t) \d t + \sigma(X_{t}^{\mathrm{D}},M_{t}^{\mathrm{D}}, \mathrm{D}_{t},t) \d B_{t} - \d \mathrm{D}_{t}, & t \geq0, \\ M_t^{\mathrm{D}}=m \wedge \inf\limits_{0\leq s\leq t} X_s^{\mathrm{D}}, &t\geq 0,\\
X_{0-}^{\mathrm{D}} = x,\quad  M^{\mathrm{D}}_{0-}=m,\quad 
\mathrm{D}_{0-}=y,
\end{cases}
\end{align}\end{small}
where $\mu$ and $\sigma$ are deterministic functions, $x$, $m$ and $y$ are real numbers satisfying $x\geq m$ and $y\geq 0$, the process $\{X_t^{\mathrm{D}}\}_{t\geq 0}$ represents the wealth process of a company and the running minimum process $\{M_t^{\mathrm{D}}\}_{t\geq 0}$ records the historical minimum value of $\{X_t^{\mathrm{D}}\}_{t\geq 0}$. The singular control $\mathrm{D}=\{\mathrm{D}_t\}_{t\geq0}$ is nondecreasing and c\`{a}dl\`{a}g with initial $\mathrm{D}_{0-}=y$. In the dividend problem, it represents the cumulative dividends paid.

As $\{M_t^{\mathrm{D}}\}_{t\geq 0}$ records the company’s minimal wealth level, the admissible set is naturally defined as $\mathcal{Q} = \{(x, m, y, t) \in \mathbb{R} \times \mathbb{R} \times [0, \infty) \times [0, \infty) \mid x \geq m\}$. Let $a: \mathbb{R} \to \mathbb{R}$ be a continuous function. We assume that bankruptcy depends on the running minimum process (and hence on the wealth process) and introduce the liquidation region $\mathcal{R}_a = \{(x, m, y, t) \in \mathcal{Q} \mid x < a(m)\}$.
The action region is then defined as $\mathcal{Q}_a := \mathcal{Q} \setminus \mathcal{R}_a = \{(x, m, y, t) \in \mathcal{Q} \mid x \geq a(m)\}$, which is clearly a closed set.
The bankruptcy time is defined as\begin{small}
$$
\tau_{t,a}^{\mathrm{D}} := \inf_{s \ge t} \, \{ X_s^{\mathrm{D}} \le a(M_s^{\mathrm{D}}) \} = \inf_{s \ge t} \, \{ (X_s^{\mathrm{D}},M_s^{\mathrm{D}},\mathrm{D}_s,s) \in \overline{\mathcal{R}_a} \},
$$\end{small}
which is an $\{\mathcal{F}_s\}_{s \ge t}$-stopping time. Bankruptcy is triggered once the wealth process $\{X_s^{\mathrm{D}}\}_{s\geq t}$ reaches the liquidation region and the worst performance is refreshed. When the initial time and the control are clear from the context, we simply write $\tau_{t,a}^{\mathrm{D}}$ as $\tau_a$.

A large number of studies indicate that a company's past performance tends to influence its future value; see, e.g., \citet*{Lintner1956} and \citet*{DO2008}. Therefore, we consider a financial problem involving a cost functional related to the historical minimum return, which is characterized by the running minimum process. Let $\E_{x,m,y,t}[\cdot]$ denote the expectation conditioned on the event $X_{t-}^{\mathrm{D}}=x$, $M_{t-}^{\mathrm{D}}=m$, and $\mathrm{D}_{t-}=y$. The payoff functional is given by
\begin{small}
\begin{align}\label{eq:payoff}
\begin{array}{r@{\,}l}
J(x,m,y,t;\mathrm{D}) = & \! \E_{x,m,y,t}\!\left[\int_t^{\tau_{t,a}^\mathrm{D}}\beta(r-t)H(X_r^{\mathrm{D}},M_r^\mathrm{D},\mathrm{D}_r,r)\d r\right] \\[1.5pt]
& + \E_{x,m,y,t}\!\left[\int_t^{\tau_{t,a}^{\mathrm{D}}}\beta(r-t)c(X_r^{\mathrm{D}},M_r^{\mathrm{D}},\mathrm{D}_r,r)\diamond \d \mathrm{D}_r\right] \\[1.5pt]
& + \E_{x,m,y,t}\!\left[\int_t^{\tau_{t,a}^{\mathrm{D}}}\beta(r-t)h(M_r^{\mathrm{D}},\mathrm{D}_r,r)\square \d M_r\right],
\end{array}
\end{align}
\end{small}
where
\begin{small}
\begin{align*}
&\int_{t}^{T} \beta(r-t) \, g\big( X_{r}^{\mathrm{D}}, M_{r}^{\mathrm{D}}, \mathrm{D}_r,r \big) \diamond \mathrm{d}\mathrm{D}_{r} := \int_{t}^{T} \beta(r-t) \, g\big( X_{r}^{\mathrm{D}}, M_{r}^{\mathrm{D}},\mathrm{D}_r,r \big) \, \mathrm{d}\mathrm{D}_{r}^{c} \\[-0.5ex]
& \quad + \sum_{\substack{t\leq r \leq T \\ \Delta\mathrm{D}_{r} \neq 0}} \beta(r-t) \int_{0}^{(X_{r-}^{\mathrm{D}} - M_{r-}^{\mathrm{D}}) \wedge \Delta\mathrm{D}_{r}} g\big( X_{r-}^{\mathrm{D}} - u, M_{r-}^{\mathrm{D}} ,\mathrm{D}_{r-}+u,r\big) \, \mathrm{d}u \\[-0.5ex]
& \quad + \sum_{\substack{t\leq r \leq T \\ \Delta\mathrm{D}_{r} \neq 0}} \beta(r-t) \int_{(X_{r-}^{\mathrm{D}} - M_{r-}^{\mathrm{D}})\wedge \Delta\mathrm{D}_r}^{\Delta\mathrm{D}_{r}}  g\big( X_{r-}^{\mathrm{D}} - u, X_{r-}^{\mathrm{D}} - u,\mathrm{D}_{r-}+u,r \big) \, \mathrm{d}u,
\end{align*}
\end{small}
and
\begin{small}
 \begin{align*}
&\int_{t}^{T} \beta(r-t) \, g\big( M_{r}^{\mathrm{D}}, \mathrm{D}_r,r \big) \square \mathrm{d}M_{r}:= \int_{t}^{T} \beta(r-t) \, g\big(  M_{r}^{\mathrm{D}},\mathrm{D}_r,r \big) \, \mathrm{d}M_{r}^{c} \\ 
&- \sum_{\substack{t\leq r \leq T \\ \Delta\mathrm{D}_{r} \neq 0}} \beta(r-t) \int_{(X_{r-}^{\mathrm{D}} - M_{r-}^{\mathrm{D}})\wedge \Delta\mathrm{D}_r}^{\Delta\mathrm{D}_{r}}  g\big( X_{r-}^{\mathrm{D}} - u,\mathrm{D}_{r-}+u,r \big) \, \mathrm{d}u
\end{align*}   
\end{small}
for any $T\in[0,\infty]$ and $t<T$, see, e.g.,  \citet*{FR2025}. In addition, we define
\begin{small}
\begin{align}\label{eq:leftint}
\int_{t}^{T-} \beta(r-t) \, g\big( X_{r}^{\mathrm{D}}, M_{r}^{\mathrm{D}}, \mathrm{D}_r,r \big) \diamond \mathrm{d}\mathrm{D}_{r}=\lim_{t'\to T-}\int_{t}^{t'} \beta(r-t) \, g\big( X_{r}^{\mathrm{D}}, M_{r}^{\mathrm{D}}, \mathrm{D}_r,r \big) \diamond \mathrm{d}\mathrm{D}_{r},
\end{align}
\end{small}
and 
$\int_{t}^{T-} \beta(r-t) \, g\big( M_{r}^{\mathrm{D}}, \mathrm{D}_r,r \big) \square \mathrm{d}{M}_{r}$ can be similarly defined. This form of payoff functional encompasses a wide range of singular control problems arising in applications, such as optimal dividends and irreversible investment; see, e.g., \citet*{KL2001} and \citet*{CGCY2021}.
\begin{remark}
These integrals are defined in a ``path averaging" manner, implying that the benefits from a singular control action at a given time are realized smoothly. If we impose the conditions $g(x,m,y,t)=g(x-u,m,y+u,t)$  for  $0\leq u \leq x-m$, and $g(x,m,y,t)=g(x-u,x-u,y+u,t)$ for $u>x-m$, the definitions coincide with those in \citet*{LLY2024}. These conditions reflect that no additional gain can be achieved by splitting a control action into multiple operations within an infinitesimal time interval, compared to a single operation of the same total magnitude.
\end{remark}
We now state the standing assumptions on the wealth process and the payoff functional, which are maintained throughout the paper.
\begin{assumption}\label{as:allpaper}
We assume the following: \\
(a) (The linear growth condition.) For any $x,m\in \R$ satisfying $m\leq x$ and $y,t\in[0,\infty)$, we have
$$|\mu(x,m,y,t)|+|\sigma(x,m,y,t)|\leq K(1+|x|+|m|+y+t).$$
(b) (The Lipschitz condition.) There exists $L>0$ such that for $\varphi=\mu,\sigma$ and for any $x_1,x_2,m_1,m_2\in \R$ satisfying $m_1\leq x_1$ and $m_2\leq x_2$, $y_1,y_2\in[0,\infty)$ and $t\in [0,\infty)$, we have
$$|\varphi(x_1,m_1,y_1,t)-\varphi(x_2,m_2,y_2,t)|\leq L(|x_1-x_2|+|m_1-m_2|+|y_1-y_2|).$$
(c) The discount function $\beta\in C^1[0,\infty)$ is a decreasing function satisfying $\beta(0)=1$.\\
(d) Functions $H$, $c$, and $h$ are continuous with respect to $(x,m,y,t)$.
\end{assumption}
We then give the definitions of admissible singular control, admissible singular control law, and equilibrium singular control law.
\begin{definition}\label{def:singular}
Given $(x,m,y,t)\in \mathcal{Q}$, we call $\{\mathrm{D}_r\}_{r\geq t}$ an admissible singular control if it is $\{\F_r\}_{r\geq t}$-adapted, nondecreasing, c\`{a}dl\`{a}g and satisfies the following (a)-(c):\\
(a) \quad $\mathrm{D}_{t-}=y$ and the stochastic equation\begin{small}
\begin{align}\label{eq:X}
\begin{cases}
\d X_{r}^{\mathrm{D}} = \mu(X_{r}^{\mathrm{D}}, M_r^{\mathrm{D}}, \mathrm{D}_{r},r) \d r + \sigma(X_{r}^{\mathrm{D}}, M_r^{\mathrm{D}}, \mathrm{D}_{r},r) \d B_{r} - \d \mathrm{D}_{r}, & r \geq t, \\M_r^{\mathrm{D}}=m \wedge \inf\limits_{t\leq s\leq r} X_s^{\mathrm{D}}, &r\geq t,\\
X_{t-}^{\mathrm{D}} = x,\quad M^{\mathrm{D}}_{t-}=m,
\end{cases}
\end{align}\end{small}
has a unique strong solution $(\{X_r^{\mathrm{D}}\}_{r\geq t} ,\{M_r^{\mathrm{D}}\}_{r\geq t})$.\\
(b) \quad The solution $(\{X_r^{\mathrm{D}}\}_{r\geq t} ,\{M_r^{\mathrm{D}}\}_{r\geq t})$ in (a) satisfies that if $(X_{r-}^{\mathrm{D}},M_{r-}^{\mathrm{D}},\mathrm{D}_{r-},r)\in \mathcal{Q}_a, \forall t\leq r\leq T$, then $(X_{T}^{\mathrm{D}},M_{T}^{\mathrm{D}},\mathrm{D}_{T},T)\in \mathcal{Q}_a$.\\
(c)\quad For any $t\geq s$, Eq. (\ref{eq:payoff}) is well-defined, namely, the results of three integral terms will not have both $\infty$ and $-\infty$.
\\
Denote by $\mathcal{D}_t$ the set of all admissible singular controls on $[t,\infty)$.
\end{definition}
If $(x,m,y,t)\in \mathcal{Q}\backslash \mathcal{Q}_a$, we have $J(x,m,y,t;\mathrm{D})=0$ for any $\mathrm{D}\in \mathcal{D}_t$ by the definition of $\tau_a^{\mathrm{D}}$ and Eq. (\ref{eq:payoff}). In the dividend problem, this means that, once the company goes bankrupt, it ceases operations and stops paying dividends. We focus on the control within the action region according to Condition (b).
\begin{proposition}\label{prop:solution1}
For any $\{\F_r\}_{r\geq t}$-adapted, nondecreasing, and c\`{a}dl\`{a}g process $\mathrm{D}=\{\mathrm{D}_r\}_{r\geq t}$ that satisfies\footnote{This moment condition can be omitted if we have $|\varphi(x,m,y,t)|\leq K(1+|x|+|m|+t)$ for $\varphi=\mu$ and $\sigma$.} \begin{small}$\E_{x,m,y,t}[(\mathrm{D}_{t+r}-\mathrm{D}_{t-})^2]<\infty$\end{small} for any $r>0$, the SDE  (\ref{eq:X}) has a unique strong solution $(X^{\mathrm{D}},M^{\mathrm{D}})$.
\end{proposition}
\begin{proof}
See Appendix \ref{app:1}.
\end{proof}
If $H$, $c$, and $-h$ are nonnegative, then any process $\mathrm{D}$ satisfying conditions (a) and (b) of Definition \ref{def:singular} is admissible.
\begin{definition}\label{def:sing2}
Suppose that $\Xi:=(W^{\Xi},P^{\Xi})$ is a partition of $\mathcal{Q}$, $W^{\Xi}$ is a relatively open subset of $\mathcal{Q}$ denoting the waiting region, and $P^{\Xi}$ is a closed set denoting the purchasing region. We call $\Xi$ an admissible singular control law if the following conditions hold:\\
(a) \quad Given any initial $(x,m,y,t)\in \mathcal{Q}$, the Skorokhod reflection problem\begin{small}
\begin{align}\label{eq:Skorokhod}
\begin{cases}
\d X_{r}^{\mathrm{D}} = \mu(X_{r}^{\mathrm{D}}, M_r^{\mathrm{D}}, \mathrm{D}_{r},r) \, \d r + \sigma(X_{r}^{\mathrm{D}}, M_r^{\mathrm{D}}, \mathrm{D}_{r},r) \, \d B_{r} - \d\mathrm{D}_{r}, & r \geq t, \\[6pt]
(X_{r}^{\mathrm{D}}, M_r^{\mathrm{D}}, \mathrm{D}_{r},r) \in \overline{W^{\mathrm{D}}}, & r \geq t, \\[6pt]
\mathrm{D}_{r} = y + \int_{t}^{r} 1_{\{(X_{u}^{\mathrm{D}}, M_u^{\mathrm{D}}, \mathrm{D}_{u},u) \in P^{\mathrm{D}}\}} \, \d\mathrm{D}_{u}, & r \geq t, \\[6pt]
X_{t-}^{\mathrm{D}} = x,\quad M_{t-}^{\mathrm{D}}=m, \quad \mathrm{D}_{t-} = y &
\end{cases}
\end{align}\end{small}
has a unique strong solution $(X^{\mathrm{D}},M^{\mathrm{D}},\mathrm{D})=(X^{x,m,y,t,\mathrm{\Xi}},M^{{x,m,y,t,\mathrm{\Xi}}},\mathrm{D}^{x,m,y,t,\mathrm{\Xi}})$. We call $\{\mathrm{D}_r\}_{r\geq t}$ the singular control generated by $\Xi$ at $(x,m,y,t)$.\\
(b)\quad For any $s\leq t$, the processes $\{\mathrm{D}_{r}\}_{r\geq t}$,$ \{X_r^{\mathrm{D}}\}_{r\geq t}$ and $\{M_r^{\mathrm{D}}\}_{r\geq t}$ satisfy 
\begin{small}
\begin{align*}
&\E_{x,m,y,t}\Bigg[\int_t^{\tau_a^\mathrm{D}}\left|\beta(r-s)H(X_r^\mathrm{D},M_r^\mathrm{D},\mathrm{D}_r,r)\right|\d r+\int_t^{\tau_a^\mathrm{D}}\left|\beta(r-s)c(X_r^\mathrm{D},M_r^\mathrm{D},\mathrm{D}_r,r)\right|\diamond \d \mathrm{D}_r\\&\quad -\int_t^{\tau_a^\mathrm{D}}\left|\beta(r-s)h(M_r^\mathrm{D},\mathrm{D}_r,r)\right|\square \d M_r^\mathrm{D} \Bigg]<\infty.
\end{align*}
\end{small}
(c)\quad The singular control $\{\mathrm{D}_r\}_{r\geq t}$ generated by $\Xi$ satisfies Condition (b) in Definition \ref{def:singular}.
\end{definition}
Figure \ref{fig:WP} visually demonstrates all the various regions defined above. For clarity, we present the case where the regions are independent of the singular control value $y$ and the time $t$.
\begin{figure}[!ht]
    \centering    \includegraphics[width=0.75\linewidth]{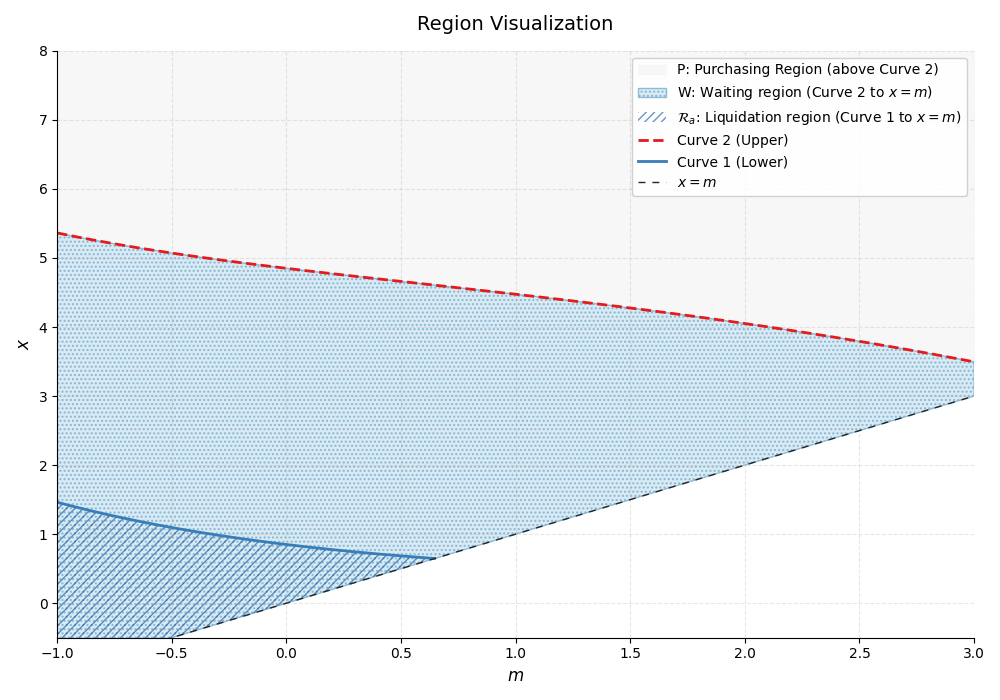}
    \caption{Region Visualization. The admissible set $\mathcal{Q}$ represents the region above $x=m$, and the action region $\mathcal{Q}_a$ represents the region above Curve 1.}
    \label{fig:WP}
\end{figure}
We next prove the existence and uniqueness of the solution to a class of Skorokhod reflection problems related to the running minimum process, where the boundary of $W$ is Lipschitz continuous. Following a possible jump at time $t$, the singular control $\mathrm{D}$ evolves in the closure of the waiting region. Hence, we only consider the cases where the initial condition $(X_{0-}, M_{0-},\mathrm{D}_{0-}) = (x_0, m_0,y_0) \in \overline{W}$.
\begin{proposition}\label{prop:Skorokhod}
Assume that the waiting region $ W$ is expressed in the form of $W=\{(x,m,y,t)| m\leq x<b(m,y,t),t\geq 0\}$, where the function $b$ is continuous and there exist $L_m,L_d\in(0,1)$ satisfying $L_m+L_d<1$ such that $b$ is $L_m$-Lipschitz with respect to the first component $m$ and is $L_d$-Lipschitz with respect to the second component $y$.
Then, \\
(a) The Skorokhod reflection problem\begin{small}
\begin{align}\label{eq:Skorokhod2}
\begin{cases}
\d X_{t}^{\mathrm{D}} = \mu(X_{t}^{\mathrm{D}}, M_t^{\mathrm{D}},\mathrm{D}_t,t) \, \d t + \sigma(X_{t}^{\mathrm{D}}, M_t^{\mathrm{D}},\mathrm{D}_t,t) \, \d B_{t} - \d\mathrm{D}_{t}, & t \geq 0, \\[6pt]
(X_{t}^{\mathrm{D}}, M_t^{\mathrm{D}},\mathrm{D}_t,t) \in \overline{W}, & t \geq 0, \\[6pt]
\mathrm{D}_{t} =  \int_{0}^{t} 1_{\{(X_{u}^{\mathrm{D}}, M_u^{\mathrm{D}},\mathrm{D}_u,u) \in P\}} \, \d\mathrm{D}_{u}, & t \geq 0, \\[6pt]
X^{\mathrm{D}}_{0-}=x_0,\quad M^{\mathrm{D}}_{0-}=m_0,\quad \mathrm{D}_{0-}=y_0&
\end{cases}
\end{align}\end{small}
with initial condition $(X_{0-}, M_{0-},\mathrm{D}_{0-}) = (x_0, m_0,y_0)  \in \overline{W}$, admits a unique strong solution $\left(\{X_t\}_{t\geq 0},\{ M_t\}_{t\geq 0},\{ \mathrm{D}_t\}_{t\geq 0}\right)$.\\
(b) Moreover, we assume the following conditions:\\
(1) The liquidation region $\mathcal{R}$ is a subset of the waiting region $W$.\\
(2) The functions $H$, $h$, $c$, $\mu$, and $\sigma$ are bounded on $\mathcal{Q}_a$, and $b$ is Lipschitz continuous with respect to the third component $t$.\\
(3) There exist constants $C_1>0$ and $C_2>1$ such that $\beta(t)\leq C_1(1+t)^{-C_2}$ holds for any $t\geq 0$.\\
Then $\Xi=(W^\Xi,P^\Xi)$ is an admissible singular control law.
\end{proposition}
\begin{proof}
See Appendix \ref{app:2}.
\end{proof}
\begin{remark}
We explain the condition $L_m + L_d < 1$. Consider a point $(x, m, y, t)$ on the boundary of $W$ satisfying $x = b(m, y, t)$. If $L_m > 1$, the waiting region $W$ may become multiply connected, which may result in a complicated and discontinuous situation for $\mathrm{D}$. If $L_d > 1$, the boundary moves faster than the control can respond, so that the process fails to track the boundary, leading to the non-existence of a strong solution. The condition $L_m + L_d < 1$ accounts for the joint influence of both the above factors.
It should be emphasized that the well-posedness for the fully general class of reflection problems is still open.
\end{remark}
Then we give the definition of equilibrium in our time-inconsistent singular control framework. 
\begin{definition}\label{def:equi}
Let $\hat{\Xi}$ be an admissible singular control law and $\mathrm{D}^{x,m,y,t,\hat{\Xi}}$ be the singular control generated by $\hat{\Xi}$. For any initial $(x,m,y,t)\in \mathcal{Q}_a$, we call $\hat{\Xi}$ an equilibrium singular control law and $J(x,m,y,t;\mathrm{D}^{x,m,y,t,\hat{\Xi}})$ the corresponding equilibrium value function if for any control $\mathrm{L}\in\mathcal{D}_t$ satisfying $\mathrm{L}_{t-}=y$ and $\mathrm{L}_{(t+h)-}-\mathrm{L}_t\leq Ch$ a.s. for some $C>0$, we have
\begin{align}\label{eq:equilibrium}
\lim \sup_{h\to 0+} \frac{J(x,m,y,t;\mathrm{D}^h)-J(x,m,y,t;\mathrm{D}^{x,m,y,t,\hat{\Xi}})}{h}\leq 0,
\end{align}
where $\mathrm{D}^h$ is defined by
\begin{align}\label{eq:D^h}
\mathrm{D}_r^h=
\begin{cases}
\mathrm{L}_r, & r \in [t, t+h), \\[6pt]
\mathrm{D}_r^{X_{(t+h)-}^{\mathrm{D}^h},M_{(t+h)-}^{\mathrm{D}^h},\ \mathrm{D}^h_{(t+h)-},t+h,\ \hat{\Xi}}, & r \in [t+h, \infty).
\end{cases}
\end{align}
We call the control generated by the equilibrium singular control law the equilibrium control.
\end{definition}
\begin{remark}\label{remark:equilibrium}
If the bankruptcy caused by the perturbation happens during $[t,t+h)$, the control after it has no impact on the payoff functional, whose increment after time $t$ is equal to $0$. The condition $\mathrm{L}_{(t+h)-}-\mathrm{L}_t\leq Ch$ a.s. is called the $O(h)$-condition. Although the $O(h)$-condition is relatively weak, it fails to capture the singular behavior of the control, such as the infinite initial slope exhibited by $\sqrt{t}$. We can extend Definition \ref{def:equi} to the cases without $O(h)$-condition to obtain a stronger equilibrium definition, provided some specific conditions hold; see Proposition \ref{prop:f^s}.
\end{remark}
\begin{remark}
Compared with \citet*{dai2024} and \citet*{CZ2025}, where the perturbed control form is directly determined as a function of $\omega\in \Omega$, our equilibrium determined by a Skorokhod reflection problem can be defined for the cases where the payoff functional is related to the control value and the running minimum. Interestingly, when both definitions are applicable, the corresponding system of HJB equations for the equilibrium is identical.
\end{remark}
Based on the above definitions, we provide a verification theorem and analyze the extended HJB equation system in the next section.
\section{The Extended HJB Equation System}\label{sec:fin}
In this section, we present the verification theorem, which serves as a sufficient condition for the existence of an equilibrium. We first introduce an important function sequence $\{f^s\}_{s\in [0,t]}$, which can be explained as the value of the payoff functional corresponding to the discounting starting point at a past time point in the time-inconsistent setting; see Eq. (\ref{eq:probpre}) below. In the following theorem, we mainly focus on its Conditions (1) and (3), while Condition (4) is standard and Conditions (2) and (5) are technical conditions for the Fatou lemma and the dominated convergence theorem (abbr. DCT) which can be replaced with more specific ones, for example, Conditions (1) and (2) in Proposition \ref{prop:f^s}. In the following, we abbreviate $\E_{x,m,y,t}[\cdot]$ to $\E[\cdot]$ when there is no risk of ambiguity.
\begin{theorem}[Verification Theorem]\label{th:verify}
Given a function $V:\mathcal{Q}\to \R$ and a family of functions $\{f^s\}_{s\in [0,t]}$, where $f^s:\mathcal{Q}\to \R$, define 
\begin{small}
\begin{align*}
&f(x,m,y,t,s) := f^s(x,m,y,t),   \\
&\mathcal{A} \varphi(x,m,y,t) := \varphi_t(x,m,y,t) + \mu(x,m,y,t)\varphi_x(x,m,y,t) \\
&\qquad + \frac12 \sigma^2(x,m,y,t)\varphi_{xx}(x,m,y,t),\quad \forall \varphi :\mathcal{Q} \to \R, \\
&W := \{(x,m,y,t)\in\mathcal{Q} \,|\, c(x,m,y,t)-V_x(x,m,y,t)+V_y(x,m,y,t)<0\},\\
&P := \{(x,m,y,t)\in\mathcal{Q} \,|\, c(x,m,y,t)-V_x(x,m,y,t)+V_y(x,m,y,t)=0\}.
\end{align*}
\end{small}
Assume that we have the following properties:\\
(1) The boundary $\partial W$ is represented as a continuous function in the form of $x=b(m,y,t)$. Moreover, for any $s\in [0,t]$, we have $V\in C^{2,1,1,1}(\overline{W})\cap C^{2,1,1,1}(\overline{P})\cap C^{1,1,0,0}(\mathcal{Q}_a)$, $f^s\in C^{2,1,1,1}(\overline{W})\cap C^{1,1,1,1}(\overline{P})\cap C^{1,1,0,0}(\mathcal{Q}_a)$, and $f$ is $C^1$ with respect to the fifth component $s$ in $\mathcal{Q}_a$.\\
(2) For any $\mathrm{D}^h\in \mathcal{D}_t$,
there exists a constant $\bar{h}>0$ such that for any $0<h_0<\bar{h}$, we have
\begin{small}
\begin{align}\label{eq:intcondition}
&\mathbb{E}_{x,m,y,t} \bigg[ \sup_{\substack{r \in [t, t+h_0), \\ u \in [0, \Delta\mathrm{D}_r]}} \Big|\varphi(X_{r-}-u, M_{r-}-u, \mathrm{D}_{r-}+u, r) \Big| \bigg] < \infty,
\end{align}
\end{small}
for $\varphi=\mathcal{A}V$, $V_y-V_x$ and $V_m$,
\begin{small}
\begin{align}\label{eq:intcondition2}
&\mathbb{E}_{x,m,y,t} \bigg[ \sup_{\substack{r \in [t, t+h_0), \\ u \in [0, \Delta\mathrm{D}_r]}} \Big| f_s(X_{(t+h_0)-}-u, M_{(t+h_0)-}-u, \mathrm{D}_{(t+h_0)-}+u, t+h_0, r) 
 \Big| \bigg] < \infty,
\end{align}
\end{small}
and for $\psi=H$ and $c$,
\begin{small}
\begin{align}\label{eq:intcondition3}
\mathbb{E}_{x,m,y,t} \bigg[ \sup_{\substack{r \in [t, t+h_0), \\ u \in [0, \Delta\mathrm{D}_r]}} \Big| &\beta(r-t)\psi(X_{r-}-u,M_{r-}-u,\mathrm{D}_{r-}+u,r)  \\[-0.8ex]
&+\beta(r-t)h(M_{r-}-u,\mathrm{D}_{r-}+u,r) \Big| \bigg] < \infty. \notag
\end{align}
\end{small}\\
(3) $V(x,m,y,t)$ and $f^s(x,m,y,t)$ satisfy
\begin{small}
\begin{align}
&\max \{\mathcal{A}V(x,m,y,t)+H(x,m,y,t)-\mathcal{A}f(x,m,y,t,t)+\mathcal{A}f^t(x,m,y,t), \label{eq:ver1}\\ &c(x,m,y,t)-V_x(x,m,y,t)+V_y(x,m,y,t)\}=0,\quad \forall(x,m,y,t)\in \mathcal{Q}_a\setminus \partial W,  s\in[0,t],\notag\\
&V_m(x,m,y,t)|_{x=m}=-h(m,y,t), \quad a(m)\leq m,\label{eq:ver2} \\
&V(a(m),m,y,t)=0,\quad a(m)>m, \label{eq:ver3}\\
&\mathcal{A}f^s(x,m,y,t)+\beta(t-s)H(x,m,y,t)=0,\quad\forall (x,m,y,t)\in \overline{W},\label{eq:ver4} \\
&\beta(t-s)c(x,m,y,t)-f^s_x(x,m,y,t)+f^s_y(x,m,y,t)=0,\quad \forall (x,m,y,t)\in P,\label{eq:ver5}\\
&f^s_m(x,m,y,t)|_{x=m}=-\beta(t-s)h(m,y,t),\quad a(m)\leq m,\label{eq:ver6}\\
&f^s(a(m),m,y,t)=0,\quad a(m)>m\label{eq:ver7}.
\end{align}
\end{small}\\
(4) $\hat{\Xi}=(W,P)$ is an admissible singular control law.\\
(5)  Denote by $\hat{\mathrm{D}}$ the singular control generated by $\hat{\Xi}$. For $\varphi= V,f^s$, and $\hat{\mathrm{D}}$, we have
\begin{small}
\begin{align}\label{eq:dominate}
\mathbb{E}_{x,m,y,t} \bigg[ \sup_{r \in [t, \tau_a^{\hat{\mathrm{D}}}]} \Big| \varphi(X_r^{\hat{\mathrm{D}}}, M_r^{\hat{\mathrm{D}}}, \hat{\mathrm{D}}_r, r) \Big| + \sup_{r \in [t, \tau_a^{\hat{\mathrm{D}}}]} \Big| f(X_r^{\hat{\mathrm{D}}}, M_r^{\hat{\mathrm{D}}}, \hat{\mathrm{D}}_r, r, r) \Big| \bigg] < \infty,
\end{align}
\end{small}
\begin{small}
\begin{align}\label{eq:dominate2}
\limsup_{T \to \infty} \mathbb{E}_{x,m,y,t} \Big[ \big( \big| \varphi(X_T^{\hat{\mathrm{D}}}, M_T^{\hat{\mathrm{D}}}, \hat{\mathrm{D}}_T, T) \big| + \big| f(X_T^{\hat{\mathrm{D}}}, M_T^{\hat{\mathrm{D}}}, \hat{\mathrm{D}}_T, T, T) \big| \big) \cdot \mathbf{1}_{\{T < \tau_a^{\hat{\mathrm{D}}}\}} \Big] = 0.
\end{align}
\end{small}
hold for any $t\geq s$.
Besides, for any $t\geq 0$ and $\mathrm{D}\in \mathcal{D}_t$, there exists $h_0>0$ such that 
\begin{small}
\begin{align}\label{eq:Dforstep3}
\int_t^{t+h_0}(\sigma(X_r^{\mathrm{D}},M_r^{\mathrm{D}},\mathrm{D}_r,r)V_x(X_r^{\mathrm{D}},M_r^{\mathrm{D}},\mathrm{D}_r,r))^2\d r <\infty.
\end{align}
\end{small}
Then $\hat{\Xi}=(W,P)$ is an equilibrium singular control law and $V$ is the corresponding equilibrium value function. Moreover, $f$ has the following probabilistic interpretation
\begin{small}
\begin{align}\label{eq:probpre}
f(x,m,y,t,s) &={}\mathbb{E}_{x,m,y,t}\left[\int_t^{\tau_a^{\hat{\mathrm{D}}}} \beta(r-s) H(X_r^{\hat{\mathrm{D}}}, M_r^{\hat{\mathrm{D}}}, \hat{\mathrm{D}}_r, r) \, \mathrm{d} r\right]  \\&\quad+ \mathbb{E}_{x,m,y,t}\left[\int_t^{\tau_a^{\hat{\mathrm{D}}}} \beta(r-s) c(X_r^{\hat{\mathrm{D}}}, M_r^{\hat{\mathrm{D}}}, \hat{\mathrm{D}}_r, r) \diamond \mathrm{d} \hat{\mathrm{D}}_r\right]\notag \\&\quad + \mathbb{E}_{x,m,y,t}\left[\int_t^{\tau_a^{\hat{\mathrm{D}}}} \beta(r-s) h(M_r^{\hat{\mathrm{D}}}, \hat{\mathrm{D}}_r, r) \square \mathrm{d} M_r^{\hat{\mathrm{D}}}\right].\notag
\end{align}
\end{small}
\end{theorem}
\begin{proof}
\textbf{Step 1.} We first prove the probabilistic representation of $f^s$. Define the stopping time $T_n=\inf\{T\geq t: \int_t^T\sigma^2(X_r^{\hat{\mathrm{D}}},M_r^{\hat{\mathrm{D}}},\hat{\mathrm{D}}_r,r)[f_x(X_r^{\hat{\mathrm{D}}},M_r^{\hat{\mathrm{D}}},\hat{\mathrm{D}}_r,r)]^2\d r \geq n\},n\in \N^*,\p_{x,m,y,t}-a.s.$. Then \begin{small}$$\left\{\int_t^{\kappa\wedge T_n \wedge \tau_a}\sigma(X_r^{\hat{\mathrm{D}}},M_r^{\hat{\mathrm{D}}},\hat{\mathrm{D}}_r,r)f_x(X_r^{\hat{\mathrm{D}}},M_r^{\hat{\mathrm{D}}},\hat{\mathrm{D}}_r,r)\d B_r\right\}_{\kappa\in [t,\infty)}$$\end{small} is a martingale. Noting that $(X_r^{\hat{\mathrm{D}}},M_r^{\hat{\mathrm{D}}},\hat{\mathrm{D}}_r,r)\in \overline{W}, \quad \forall r\geq t$ a.s., we obtain:  for $t\geq s$,
\begin{footnotesize}
\begin{align}\label{eq:ver_f}
&\E_{x,m,y,t} \! \left[f^s(X_{\kappa\wedge T_n \wedge \tau_a}^{\hat{\mathrm{D}}},M_{\kappa\wedge T_n \wedge \tau_a}^{\hat{\mathrm{D}}},\hat{\mathrm{D}}_{\kappa\wedge T_n \wedge \tau_a},{\kappa\wedge T_n \wedge \tau_a})\right]\!-\!f^s(x,m,y,t)  \\[-1.2ex]
&=\E_{x,m,y,t} \Bigg[\int_t^{\kappa\wedge T_n \wedge \tau_a} \!\!\!\!\!\! \mathcal{A}f^s(X_r^{\hat{\mathrm{D}}},M_r^{\hat{\mathrm{D}}},\hat{\mathrm{D}}_r,r)\d r \!+\! \int_t^{\kappa\wedge T_n\wedge\tau_a} \!\!\!\!\!\! \sigma(X_r^{\hat{\mathrm{D}}},M_r^{\hat{\mathrm{D}}},\hat{\mathrm{D}}_r,r)f_x(X_r^{\hat{\mathrm{D}}},M_r^{\hat{\mathrm{D}}},\hat{\mathrm{D}}_r,r)\d B_r \notag \\[-1.2ex]
&\quad+\!\int_t^{\kappa\wedge T_n\wedge\tau_a} \!\!\!\!\!\! f^s_m(X_r^{\hat{\mathrm{D}}},M_r^{\hat{\mathrm{D}}},\hat{\mathrm{D}}_r,r)\d (M^{\hat{\mathrm{D}}})^c_r \!+\! \int_t^{\kappa\wedge T_n\wedge\tau_a} \!\!\!\!\!\! (f^s_y(X_r^{\hat{\mathrm{D}}},M_r^{\hat{\mathrm{D}}},\hat{\mathrm{D}}_r,r) \!-\! f_x^s(X_r^{\hat{\mathrm{D}}},M_r^{\hat{\mathrm{D}}},\hat{\mathrm{D}}_r,r))\d \hat{\mathrm{D}}^c_r \notag \\[-0.5ex]
&\quad +\!\! \sum_{\scriptscriptstyle t\leq r\leq \kappa\wedge T_n\wedge \tau_a} \!\!\! \big(f^s(X_{r\wedge\tau_a},M_{r\wedge\tau_a},\hat{\mathrm{D}}_{r\wedge\tau_a},r\wedge\tau_a) \!-\! f^s(X_{(r\wedge\tau_a)-},M_{(r\wedge\tau_a)-},\hat{\mathrm{D}}_{(r\wedge\tau_a)-},r\wedge\tau_a)\big) \Bigg],\notag
\end{align}
\end{footnotesize}
where we have used the It\^{o}-Tanaka-Meyer formula; see \citet*{KS1991}.
Let $\eta = r\wedge\tau_a$. By direct calculation, we have
\begin{small}
\begin{align}\label{eq:verf2}
&f^s(X^{\hat{\mathrm{D}}}_{\eta},M^{\hat{\mathrm{D}}}_{\eta},\hat{\mathrm{D}}_{\eta},\eta)-f^s(X^{\hat{\mathrm{D}}}_{\eta-},M^{\hat{\mathrm{D}}}_{\eta-},\hat{\mathrm{D}}_{\eta-},\eta) \\[-0.5ex]
&=\int_0^{X^{\hat{\mathrm{D}}}_{\eta-}-M^{\hat{\mathrm{D}}}_{\eta-}} \!\!\!\!\!\!(f_y^s-f^s_x)(X^{\hat{\mathrm{D}}}_{\eta-}-u,M^{\hat{\mathrm{D}}}_{\eta-},\hat{\mathrm{D}}_{\eta-}+u,\eta)\d u\notag \\[-0.5ex]
&\quad +\int_{X^{\hat{\mathrm{D}}}_{\eta-}-M^{\hat{\mathrm{D}}}_{\eta-}}^{X^{\hat{\mathrm{D}}}_{\eta-}-M^{\hat{\mathrm{D}}}_{\eta}} \!\!\!\!\!\!(f_y^s-f^s_x-f^s_m)(X^{\hat{\mathrm{D}}}_{\eta-}-u,X^{\hat{\mathrm{D}}}_{\eta-}-u,\hat{\mathrm{D}}_{\eta-}+u,\eta)\d u\notag \\[-0.5ex]
&=-\int_{\eta-}^{\eta}\!\!\!\beta(u-s) h(M_u^{\hat{\mathrm{D}}},\hat{\mathrm{D}}_u,u)\square \d M^{\hat{\mathrm{D}}}_u -\int_{\eta-}^{\eta}\!\!\!\beta(u-s) c(X_u^{\hat{\mathrm{D}}},M_u^{\hat{\mathrm{D}}},\hat{\mathrm{D}}_u,u)\diamond \d \hat{\mathrm{D}}_u,\notag
\end{align}
\end{small}
where we have just used the Newton-Leibniz formula, the definition of the integrals in Section \ref{sec:model}, and Eqs. (\ref{eq:ver4})-(\ref{eq:ver6}).
Then, combining Eqs. (\ref{eq:ver_f}) and (\ref{eq:verf2}) yields
\begin{small}
\begin{align}
&\mathbb{E}_{x,m,y,t} \big[f^s(X_{\kappa\wedge T_n \wedge \tau_a}^{\hat{\mathrm{D}}},M_{\kappa\wedge T_n \wedge \tau_a}^{\hat{\mathrm{D}}},\hat{\mathrm{D}}_{\kappa\wedge T_n \wedge \tau_a},{\kappa\wedge T_n \wedge \tau_a})\big]-f^s(x,m,y,t)  \\[-0.8ex]
&= -\mathbb{E}_{x,m,y,t} \Big[ \int_t^{\kappa\wedge T_n\wedge\tau_a} \!\!\!\!\!\! \beta(r-s) \big( H(X_r^{\hat{\mathrm{D}}},M_r^{\hat{\mathrm{D}}},\hat{\mathrm{D}}_r,r) \d r \notag \\[-1.2ex]
&\qquad\qquad\quad + h(M_r^{\hat{\mathrm{D}}},\hat{\mathrm{D}}_r,r) \square \d M_r^{\hat{\mathrm{D}}} + c(X_r^{\hat{\mathrm{D}}},M_r^{\hat{\mathrm{D}}},\hat{\mathrm{D}}_r,r) \diamond \d \hat{\mathrm{D}}_r \big) \Big].\notag
\end{align}
\end{small}
Taking the limit as $n\to \infty$, and using Eq. (\ref{eq:dominate}) and the DCT yield
\begin{small}
\begin{align*}
&f^s(x,m,y,t) = \mathbb{E}_{x,m,y,t} \big[f^s(X_{\kappa \wedge \tau_a}^{\hat{\mathrm{D}}},M_{\kappa \wedge \tau_a}^{\hat{\mathrm{D}}},\hat{\mathrm{D}}_{\kappa \wedge \tau_a},{\kappa \wedge \tau_a})\big] \\[-0.8ex]
&+\mathbb{E}_{x,m,y,t} \Bigg[ \int_t^{\kappa\wedge\tau_a} \!\!\!\!\!\! \beta(r-s) \Big( H(X_r^{\hat{\mathrm{D}}},M_r^{\hat{\mathrm{D}}},\hat{\mathrm{D}}_r,r) \d r + h(M_r^{\hat{\mathrm{D}}},\hat{\mathrm{D}}_r,r) \square \d M_r^{\hat{\mathrm{D}}} \\[-1.2ex]
&\hspace{10.5em} + c(X_r^{\hat{\mathrm{D}}},M_r^{\hat{\mathrm{D}}},\hat{\mathrm{D}}_r,r) \diamond \d \hat{\mathrm{D}}_r \Big) \Bigg].
\end{align*}
\end{small}
Then, letting $\kappa\to \infty$ and using Eq. (\ref{eq:dominate2}) together with Condition (b) in Definition \ref{def:sing2}, we obtain
\begin{small}
\begin{align}
f^s(x,m,y,t) &= \mathbb{E}_{x,m,y,t} \Bigg[ \int_t^{\tau_a} \!\!\! \beta(r-s) \Big( H(X_r^{\hat{\mathrm{D}}}, M_r^{\hat{\mathrm{D}}}, \hat{\mathrm{D}}_r, r) \d r \\[-1.0ex]
&\quad + h(M_r^{\hat{\mathrm{D}}}, \hat{\mathrm{D}}_r, r) \square \d M_r^{\hat{\mathrm{D}}}  + c(X_r^{\hat{\mathrm{D}}}, M_r^{\hat{\mathrm{D}}}, \hat{\mathrm{D}}_r, r) \diamond \d \hat{\mathrm{D}}_r \Big) \Bigg].\notag
\end{align}
\end{small}
Therefore, we obtain the probabilistic representation of $f$.\\
\textbf{Step 2.} We prove $V(x,m,y,t)=f(x,m,y,t,t)$. As $f$ is $C^1$ with respect to the fifth component $s$, imitating the approach in the proof of \textbf{Step 1} yields 
\begin{small}
\begin{align}\label{eq:ftt}
f(x,m,y,t,t) &=\E_{x,m,y,t}\Bigg[\int_t^{\tau_a}\mathcal{A}f(X^{\hat{\mathrm{D}}}_r,M^{\hat{\mathrm{D}}}_r,\hat{\mathrm{D}}_r,r,r)\d r+\int_t^{\tau_a}h(M_r^{\hat{\mathrm{D}}},\hat{\mathrm{D}}_r,r)\square\d M_r^{\hat{\mathrm{D}}}\\&\quad+\int_t^{\tau_a}c(X_r^{\hat{\mathrm{D}}},M_r^{\hat{\mathrm{D}}},\hat{\mathrm{D}}_r,r)\diamond \d \hat{\mathrm{D}}_r
 \Bigg],\notag
\end{align}
\end{small}
\begin{small}
\begin{align}\label{eq:Vt}
V(x,m,y,t)&=\E_{x,m,y,t}\Bigg[\int_t^{\tau_a}\mathcal{A}V(X^{\hat{\mathrm{D}}}_r,M^{\hat{\mathrm{D}}}_r,\hat{\mathrm{D}}_r,r)\d r+\int_t^{\tau_a}h(M_r^{\hat{\mathrm{D}}},\hat{\mathrm{D}}_r,r)\square\d M_r^{\hat{\mathrm{D}}}\\&\quad+\int_t^{\tau_a}c(X_r^{\hat{\mathrm{D}}},M_r^{\hat{\mathrm{D}}},\hat{\mathrm{D}}_r,r)\diamond \d \hat{\mathrm{D}}_r
 \Bigg].\notag
\end{align}
\end{small}
According to Eqs. (\ref{eq:ver1}) and (\ref{eq:ver4}), we have $\mathcal{A}f(x,m,y,t,t)=\mathcal{A}V(x,m,y,t)$ for any $(x,m,y,t)\in \overline{W}$. Combining this with Eqs. (\ref{eq:ftt}) and (\ref{eq:Vt}), we have \begin{small}
\begin{align}\label{eq:V=f}
V(x,m,y,t)=f(x,m,y,t,t)=J(x,m,y,t;\hat{\mathrm{D}}).
\end{align}\end{small}
Thus, $f^t \in C^{2,1,1,1}(P)$, which ensures Eq. (\ref{eq:ver1}) is well-defined on $P$.\\
\textbf{Step 3.} We prove the equilibrium. For any $\mathrm{L}\in \mathcal{D}_t$ with $\mathrm{L}_{t-}=y$, define\begin{small}
\begin{small}
\begin{align*}
f^{\mathrm{L}}(x,m,y,t,s) := \mathbb{E}_{x,m,y,t} \bigg[ \int_t^{\tau_a^{\mathrm{L}}}  \beta(r-s) \Big( H \d r + c \diamond \d \mathrm{L}_r + h \square \d M_r^{\mathrm{L}} \Big) (X_r^{\mathrm{L}}, M_r^{\mathrm{L}}, \mathrm{L}_r, r) \bigg].
\end{align*}
\end{small}\end{small}
To simplify the exposition, for any $\{\F_s\}_{s\geq t}$-stopping time $\tau$, we introduce a symbolic representation $(t+h)-\wedge \tau$ as follows: \\ We define $(t+h)-\wedge \tau:=(t+h)-$ if $\tau \geq t+h$ and $(t+h)-\wedge \tau:=\tau$ if $\tau < t+h$. The meaning of $(t+h)-$ is given by Eq. (\ref{eq:leftint}).
For the sake of simplicity, all the $(t+h)-$ below represent $(t+h)-\wedge \tau_a^{\mathrm{D}^h}$. 

If $\tau \leq t+h$, we have
$X_{\tau}^{\mathrm{L}}=a(M_\tau^{\mathrm{L}})$ according to Condition (b) in Definition \ref{def:singular}. Using Eqs. (\ref{eq:probpre}) and (\ref{eq:V=f}), we have
\begin{small}
\begin{align*}
&J(x,m,y,t;\mathrm{D}^h)-J(x,m,y,t;\hat{\mathrm{D}}) \\[-1.5ex]
&=\E\bigg[ f^{\mathrm{D}^h}(X_{(t+h)\text{-}}^{\mathrm{D}^h},M_{(t+h)\text{-}}^{\mathrm{D}^h},\mathrm{D}_{(t+h)\text{-}},t\!+\!h,t)-V(x,m,y,t) \\[-1.5ex]
&\quad + \int_t^{(t+h)\text{-}} \!\! \beta(r-t)H(X_r^{\mathrm{D}^h}\!,M_r^{\mathrm{D}^h}\!,\mathrm{D}^h_r,r)\d r + \int_t^{(t+h)\text{-}} \!\! \beta(r-t)c(X_r^{\mathrm{D}^h}\!,M_r^{\mathrm{D}^h}\!,\mathrm{D}^h_r,r)\diamond \d \mathrm{D}_r^h \\[-1.5ex]
&\quad + \int_t^{(t+h)\text{-}} \!\! \beta(r-t)h(M_r^{\mathrm{D}^h}\!,\mathrm{D}^h_r,r)\square \d M^{\mathrm{D}^h}_r \bigg].
\end{align*}
\end{small}
Because $f^{\mathrm{D}^h}(X_{t+h}^{\mathrm{D}^h},M_{t+h}^{\mathrm{D}^h},\mathrm{D}_{t+h}^h,t+h,s)=0$ a.s. on $\{\tau_a^{\mathrm{D}^h}\le t+h\}$, and
\begin{small}
\begin{align*}
f^{\mathrm{D}^h}(X_{(t+h)\text{-}}^{\mathrm{D}^h},M_{(t+h)\text{-}}^{\mathrm{D}^h},\mathrm{D}_{(t+h)\text{-}},t\!+\!h,s) \!=\! f(X_{(t+h)\text{-}}^{\mathrm{D}^h},M_{(t+h)\text{-}}^{\mathrm{D}^h},\mathrm{D}_{(t+h)\text{-}},t\!+\!h,s)
\end{align*}
\end{small}
 for any $0\le s\le t$, using Eq. (\ref{eq:probpre}), we have
\begin{small}
\begin{align} \label{eq:mideq}
&J(x,m,y,t;\mathrm{D}^h)-J(x,m,y,t;\hat{\mathrm{D}}) \\[-0.5ex]
&=\E\Big[ f(X_{(t+h)-}^{\mathrm{D}^h},M_{(t+h)-}^{\mathrm{D}^h},\mathrm{D}_{(t+h)-},t+h,t) - f(X_{(t+h)-}^{\mathrm{D}^h},M_{(t+h)-}^{\mathrm{D}^h},\mathrm{D}_{(t+h)-},t+h,t+h) \notag \\[-0.5ex]
&\quad - V(x,m,y,t) + V(X_{(t+h)-}^{\mathrm{D}^h},M_{(t+h)-}^{\mathrm{D}^h},\mathrm{D}_{(t+h)-},t+h)\notag \\[-0.5ex] &\quad  + \int_t^{(t+h)-} \beta(r-t)H(X_r^{\mathrm{D}^h},M_r^{\mathrm{D}^h},\mathrm{D}^h_r,r)\d r  + \int_t^{(t+h)-} \beta(r-t)c(X_r^{\mathrm{D}^h},M_r^{\mathrm{D}^h},\mathrm{D}^h_r,r)\diamond \d \mathrm{D}_r^h\notag \\ &\quad + \int_t^{(t+h)-} \beta(r-t)h(M_r^{\mathrm{D}^h},\mathrm{D}^h_r,r)\square \d M_r^{\mathrm{D}^h} \Big]. \notag
\end{align}
\end{small}
In order to prove the equilibrium inequality Eq. (\ref{eq:equilibrium}), we first consider the cases where $\Delta \mathrm{D}_t^h=0$ (so that $X_{t}^{\mathrm{D}^h}=X_{t-}^{\mathrm{D}^h}$). The continuity of $f_s$ and Eq. (\ref{eq:ver1}) yield $\mathcal{A}V(x,m,y,t)+H(x,m,y,t)-f_s(x,m,y,t,t)\leq0,\quad \forall (x,m,y,t)\in \mathcal{Q}_a$. 
As $\partial{W}$ is continuous, applying the extended change of variable formula; see Theorem 3.1 in \citet*{PG2005} and Fatou's lemma to Eq. (\ref{eq:mideq}), we have
\begin{small}
\begin{align}\label{eq:JJ}
&\limsup_{h\to 0+}\frac{J(x,m,y,t;\mathrm{D}^h)-J(x,m,y,t;\hat{\mathrm{D}})}{h}  \\[-0.4ex]
&\leq \limsup_{h\to 0+} \E\Bigg[\int_t^{(t+h)-}\Big (\frac{\mathcal{A}V(X_r^{\mathrm{D}^h},M_r^{\mathrm{D}^h},\mathrm{D}^h_r,r) -f_s(X_{(t+h)-}^{\mathrm{D}^h},M_{(t+h)-}^{\mathrm{D}^h},\mathrm{D}^h_{(t+h)-},t+h,r)}{h} \notag \\[-0.4ex]
&\quad +\frac{\beta(r-t)H(X_r^{\mathrm{D}^h},M_r^{\mathrm{D}^h},\mathrm{D}^h_r,r)}{h}\Big )\d r \notag \\[-0.4ex]
&\quad+\int_t^{(t+h)-}\frac{V_y(X_r^{\mathrm{D}^h},M_r^{\mathrm{D}^h},\mathrm{D}^h_r,r)-V_x(X_r^{\mathrm{D}^h},M_r^{\mathrm{D}^h},\mathrm{D}^h_r,r)}{h} \diamond \d \mathrm{D}_r^h \notag \\[-0.4ex]
&\quad + \int_t^{(t+h)-}\frac{\beta(r-t)c(X_r^{\mathrm{D}^h},M_r^{\mathrm{D}^h},\mathrm{D}^h_r,r)}{h}\diamond \d \mathrm{D}_r^h \notag \\[-0.4ex]
&\quad+\int_t^{(t+h)-}\frac{V_m(X_r^{\mathrm{D}^h},M_r^{\mathrm{D}^h},\mathrm{D}^h_r,r)+\beta(r-t)h(M_r^{\mathrm{D}^h},\mathrm{D}^h_r,r)}{h}\square \d M_r^{\mathrm{D}^h} \Bigg]. \notag
\end{align}
\end{small}
Besides, we have just used Eq. (\ref{eq:Dforstep3}) to make the local martingale a martingale. Thus, by Fatou's lemma, we have
\begin{small}
\begin{align}\label{eq:J1}
&\limsup_{h\to 0+}\frac{J(x,m,y,t;\mathrm{D}^h)-J(x,m,y,t;\hat{\mathrm{D}})}{h} \\[-0.6ex]
& \leq \mathbb{E}\Bigg[ \lim_{h\to 0+} \frac{1}{h} \bigg( \int_t^{(t+h)\text{-}} \!\! \Big( \mathcal{A}V(X_t^{\mathrm{D}^h}\!,M_t^{\mathrm{D}^h}\!,\mathrm{D}^h_t,t) + H(X_t^{\mathrm{D}^h}\!,M_t^{\mathrm{D}^h}\!,\mathrm{D}^h_t,t) \notag \\[-0.9ex]
&\qquad - f_s(X_{(t+h)\text{-}}^{\mathrm{D}^h}\!,M_{(t+h)\text{-}}^{\mathrm{D}^h}\!,\mathrm{D}^h_{(t+h)\text{-}},t\!+\!h,t) \Big) \mathrm{d}r \notag \\[-0.6ex]
&\quad + \int_t^{(t+h)\text{-}} \!\! \Big( V_y(X_t^{\mathrm{D}^h}\!,M_t^{\mathrm{D}^h}\!,\mathrm{D}^h_t,t) \!-\! V_x(X_t^{\mathrm{D}^h}\!,M_t^{\mathrm{D}^h}\!,\mathrm{D}^h_t,t) \!+\! c(X_t^{\mathrm{D}^h}\!,M_t^{\mathrm{D}^h}\!,\mathrm{D}^h_t,t) \Big) \diamond \mathrm{d}\mathrm{D}_r^h \notag \\[-0.6ex]
&\quad + \int_t^{(t+h)\text{-}} \!\! \Big( V_m(X_t^{\mathrm{D}^h}\!,M_t^{\mathrm{D}^h}\!,\mathrm{D}^h_t,t) \!+\! h(M_t^{\mathrm{D}^h}\!,\mathrm{D}^h_t,t) \Big) \square \mathrm{d}M_r^{\mathrm{D}^h} \bigg) \Bigg] \leq 0, \notag
\end{align}
\end{small}
where the integrability condition of Fatou's lemma 
in the first equality is satisfied due to Eqs. (\ref{eq:intcondition})-(\ref{eq:intcondition3}) and the $O(h)$-condition in Definition \ref{def:equi}. The last inequality holds due to Eq. (\ref{eq:ver2}), the fact that $(t+h)-\wedge \tau\leq t+h$, and the right-continuity of $\{X_s^{\mathrm{D}^h}\}_{s\geq t}$ and $\{\mathrm{D}^h_s\}_{s\geq t}$.\\
Thus, we have proved the equilibrium for the cases where $\Delta\mathrm{D}_t^h=0$. Then, we consider $\Delta\mathrm{D}_t^h>0$.
Noting that 
\begin{small}
\begin{align*}
&V(X^{{\mathrm{D}}^h}_t\!,M^{{\mathrm{D}}^h}_t\!,\mathrm{D}^h_t,t)-V(x,m,y,t)\\
&=\int_0^{\Delta \mathrm{D}^h_t\wedge (x-m)} \!\! \big(\!-V_x(x-r,m,y+r,t)+V_y(x-r,m,y+r,t)\big)\d r \\[-0.6ex]
&\quad +\int_{\Delta \mathrm{D}^h_t\wedge(x-m)}^{\Delta \mathrm{D}^h_t} \!\! \big(\!-V_x+V_y-V_m\big)(x-r,x-r,y+r,t)\d r \\[-0.6ex]
&\leq \int_0^{\Delta \mathrm{D}^h_t\wedge (x-m)} \!\! -c(x\!-\!r,m,y\!+\!r,t)\d r\! +\!\! \int_{\Delta \mathrm{D}^h_t\wedge(x-m)}^{\Delta \mathrm{D}^h_t} \!\! \big(\!-\!c(x\!-\!r,m,y\!+\!r,t)\!+\!h(x\!-\!r,y\!+\!r,t)\big)\d r \\[-0.6ex]
& =-\int_{t-}^{t}c(X_r^{\mathrm{D}^h}\!,M_r^{\mathrm{D}^h}\!,\mathrm{D}^h_r,r)\diamond \d \mathrm{D}_r^h -\int_{t-}^{t}h(M_r^{\mathrm{D}^h}\!,\mathrm{D}^h_r,r)\square \d M_r^{\mathrm{D}^h}
\end{align*}
\end{small}
holds due to the Newton-Leibniz formula and Eq. (\ref{eq:ver1}), combining it with Eq. (\ref{eq:J1}), we  obtain that $\hat{\mathrm{D}}$ is an equilibrium.
Therefore, the proof is complete.
\end{proof}
\begin{remark}\label{remark:regularity}
Compared with previous literature such as \citet*{LLY2024}, which requires the existence of $V_{xx}$ and $f_{xx}$ everywhere\footnote{see \textbf{Step 3.1} in the proof of Theorem 3.1 in \citet*{LLY2024}}, our regularity assumptions are significantly weaker. In Theorem \ref{th:verify}, as $(X_r^{\hat{\mathrm{D}}},M_r^{\hat{\mathrm{D}}},\hat{\mathrm{D}}_r,r)\in \overline{W}$ for all $r\geq t$, it suffices that $f^s$ and $V$ belong to $C^{2,1,1,1}(\overline{W})\cap C^{1,1,1,1}(\overline{P})$ to apply the It\^{o}-Tanaka-Meyer formula in \textbf{Step 1} and \textbf{Step 2}. To derive Eqs. (\ref{eq:ver_f}), (\ref{eq:ftt}) and (\ref{eq:Vt}), we require $V$ and $f^s$ to be $C^1$ in the first variable $x$. 
Furthermore, to establish the equilibrium properties in \textbf{Step 3}, we need $V\in C^{2,1,1,1}(\overline{P})$ and $f$ to be $C^1$ in the fifth variable $s$. The latter requirement is satisfied under some integrability conditions concerning $H$, $c$, and $h$; see \citet*{LLY2024}.

We emphasize that Theorem \ref{th:verify} does not require $V$ to satisfy a second-order smooth fit condition on $\partial W$ with respect to $x$. In practice, however, this condition often holds in concrete examples, as it typically induces the associated variational inequality; see, e.g., \citet*{CZ2025}, Remark \ref{remark:connect} and Subsection \ref{subsec:rigor}. To the best of our knowledge, there is no record in the literature of the explicit solution to an equilibrium where $V$ fails to be twice continuously differentiable with respect to $x$.
\end{remark}
As noted in Remark \ref{remark:equilibrium}, the definition of equilibrium can be extended under certain specific assumptions. The corresponding sufficient conditions are provided in the following proposition. These conditions can be verified by direct calculation. Therefore, in the next section, we will employ them to establish the existence of an equilibrium.
\begin{proposition}\label{prop:f^s}
Assuming Conditions (1), (3), (4) and (5) in Theorem \ref{th:verify} hold, together with the following:\\
(1) For any $x\geq0$, $m\geq0$, $y\geq0$, $t\geq0$, and $0\le s\le t$, we have
\begin{align}\label{eq:module}
|f_s(x,m,y,t,s)|&\leq C(1+x^2+m^2+y^2+g(s)+g(t)),
\end{align}
where $g$ is bounded on $[t,t+h]$ for any $t\geq0$ and $h>0$.\\
(2)
The functions $H$, $c$ and $h$ satisfy the linear growth condition
\begin{align}\label{eq:module2}
&|H(x,m,y,t)|+|c(x,m,y,t)|+|h(m,y,t)|\leq C(1+|x|+|m|+|y|+t)
\end{align}
for any $x\geq0$, $m\geq0$, $y\geq0$, and $t\geq0$,
where $C>0$ is a constant.\\
(3) The liquidation boundary function $a$ is bounded.\\
Then $\hat{\mathrm{D}}$ satisfies $$\lim \sup_{h\to 0+} \frac{J(x,m,y,t;\mathrm{D}^h)-J(x,m,y,t;\mathrm{D}^{x,m,y,t,\hat{\Xi}})}{h}\leq 0,$$ where $h>0$ and $\mathrm{D}^h$ are defined in Eq. (\ref{eq:D^h}) for an arbitrary control $\mathrm{L}\in\mathcal{D}_t$ with $\mathrm{L}_{t-}=y$.
\end{proposition}
\begin{proof}
See Appendix \ref{app:3}.
\end{proof}
\begin{remark}\label{remark:int}
Proposition \ref{prop:f^s} provides an alternative to Condition (2) in Theorem \ref{th:verify}. Moreover, Condition (5) for the case $s=t$ is unnecessary if we have $\lim_{s\to t}f(x,m,y,t,s)=V(x,m,y,t)$. Indeed, under this assumption, Condition (5) in Theorem \ref{th:verify} holds for $s<t$, hence Eq. \eqref{eq:probpre} holds for $s<t$. 
As a result, Condition (b) in Definition \ref{def:sing2} and the DCT indicate Condition (5) in Theorem \ref{th:verify} holds for $s=t$. This result exactly corresponds to \textbf{Step 2} of Theorem \ref{th:verify}; thereafter, \textbf{Step 3} can be verified by the same argument as in the proof of Theorem \ref{th:verify}. We will use these alternative conditions in the proof of Proposition \ref{prop:spec} in Subsection \ref{subsec:rigor}.
\end{remark}
The verification theorem featured by the HJB equation system Eqs. (\ref{eq:ver2})-(\ref{eq:ver7}) provides a sufficient condition for the existence of the equilibrium. We further investigate the necessary conditions that Eqs. (\ref{eq:ver2})-(\ref{eq:ver7})  hold.
\begin{theorem}\label{th:nec}
Let $\hat{\Xi}$ be an admissible singular control law with equilibrium payoff functional $V$. Define
\begin{small}
\begin{align*}
f^s(x,m,y,t):=f(x,m,y,t,s),\; \forall s\ge0, \forall(x,m,y,t)\!\in\!\mathbb{R}\!\times\!(-\infty,x]\!\times\![0,\infty)\!\times\![s,\infty).
\end{align*}
\end{small}
where $f$ is given by Eq. (\ref{eq:probpre}). Assume that the following conditions hold.\\
(1) \begin{small}
\begin{align*}
&V \in C^{2,1,1,1}(\mathcal{Q}_a),  \ f \in C^{2,1,1,1,1}(\{(x,m,y,t,s) \mid (x,m,y,t) \in W, s \in [0,t]\}), \\[-0.6ex]
&f \in C^{1,1,1,1,1}(\{(x,m,y,t,s) \mid (x,m,y,t) \in P, s \in [0,t]\}).
\end{align*}
\end{small}\\
(2) For any $(x,m,y,t)\in\mathcal{Q}_a$, the singular control
    $\tilde{\mathrm{D}}:=\{\tilde{\mathrm{D}}_r\equiv y,\;r\ge t\}$ is admissible for Problem \eqref{eq:payoff}.\\
(3) For any $(x,m,y,t)\in\mathcal{Q}_a$, any $s\in[0,t]$, and $\mathrm{D}=\tilde{\mathrm{D}},\hat{\mathrm{D}}$, there exists $h_0:=h_0(x,m,y,t,s,\mathrm{D})$ such that
\begin{small}
\begin{align}
&\E_{x,m,y,t}\Big[\int_t^{t+h_0} \!\! \big(\sigma V_x(X_r^{\mathrm{D}},M_r^{\mathrm{D}},\mathrm{D}_r,r)\big)^2 \d r\Big] < \infty, \ \text{and} \label{eq:cond1} \\[-0.6ex]
&\E\Big[\sup_{h\in(0,h_0)} \int_t^{(t+h)\wedge\tau_a^{\mathrm{D}}} \!\!| I(r)|\d r \Big] < \infty \notag
\end{align}
\end{small}
hold, where $I(r)$ is given by \begin{small}
\begin{align*}
\frac{\mathcal{A}V(X_r^{\mathrm{D}},M_r^{\mathrm{D}},\mathrm{D}_r,r) - f_s(X_{(t+h)\text{-}}^{\mathrm{D}},M_{(t+h)\text{-}}^{\mathrm{D}},\mathrm{D}_{(t+h)\text{-}},t\!+\!h,r) + \beta(r\!-\!t)H(X_r^{\mathrm{D}},M_r^{\mathrm{D}},\mathrm{D}_r,r)}{h}.
\end{align*}
\end{small}
(4) For any $(m,t)$ such that $(x,m,y,t)\in W$  and any 
$d>0$, there exists $h>0$ such that for any $t_1\in(t,t+h)$, we have
$\E_{m,m,y,t}\big[\inf\limits_{r\in[t,t_1\wedge\tau_d)}X_r^{\hat{\mathrm{D}}}\big]<m$, where \begin{small}$$\tau_d:=\tau_a^{\hat{\mathrm{D}}} \wedge \inf\{r|r\geq t,|X^{\hat{\mathrm{D}}}-x|+|M^{\hat{\mathrm{D}}}-m|+|\hat{\mathrm{D}}-y|+|r-t|>d\}.$$\end{small}
Then the extended HJB system Eqs. (\ref{eq:ver2})-(\ref{eq:ver7}) is satisfied.
\end{theorem}
\begin{proof}
Based on  Definition of the stopping time, Eqs. (\ref{eq:ver3}) and (\ref{eq:ver7}) hold. Next, we prove Eqs. (\ref{eq:ver4})-(\ref{eq:ver5}). Let $P^\circ$ denote the interior of $P$. For any $(x,m,y,t)\in P^{\circ}$ (so that $x\neq m$), $s\leq t$ and $u>0$ such that $(x-v,m,y+v,t)\in P^{\circ}$ for all $0<v<u$, Eq. (\ref{eq:probpre}) and the definition of the new integrals yield\begin{small}
\begin{align*}
f(x,m,y,t,s)-f(x-u,m,y+u,t,s)=\int_0^u\beta(t-s)c(x-r,m,y+r,t)\,\mathrm{d}r .
\end{align*}
\end{small}
Hence, using the Newton-Leibniz formula, we have
\begin{small}
\begin{align*}
\int_0^u \!\! \beta(t-s)c(x-r,m,y+r,t)\d r = \int_0^u \!\! \big(f_x - f_y\big)(x-r,m,y+r,t,s) \d r.
\end{align*}
\end{small}
Differentiating with respect to $u$ on both sides and using the continuity of $f_x$ and $f_y$, we obtain Eq. (\ref{eq:ver5}). Moreover, applying Eq.  (\ref{eq:ver5}) and a similar argument to $(m,m,y,t)\in P$ and $(m-u,m-u,y+u,t)\in P$ gives Eq. (\ref{eq:ver6}) for $(x,m,y,t)\in P$.

According to Definition \ref{def:sing2}, for $(x,m,y,t)\in W^{\circ}$ (hence $x>m$), there exists $d>0$ such that $\hat{\mathrm{D}}$ is constant on $[t,(t+h)\wedge \tau_d]$ for any $h>0$. Using the It\^{o}-Tanaka-Meyer formula for $f^s$, we have\begin{small}
\begin{align*}
\E_{x,m,y,t}\!\Biggl[\frac{1}{h}\int_t^{(t+h)\wedge \tau_d}\!\bigl(\mathcal{A}f^s\bigr)(X_r^{\hat{\mathrm{D}}},M_r^{\hat{\mathrm{D}}},\hat{\mathrm{D}},r)
+\beta(r-s)H(X_r^{\hat{\mathrm{D}}},M_r^{\hat{\mathrm{D}}},\hat{\mathrm{D}},r)\,\mathrm{d}r\Biggr]=0,
\end{align*}
\end{small}
where Condition (3) ensures that the local martingale is a true martingale. Letting $h\to0^+$ and applying the DCT together with Conditions (1) and (3) yields Eq. (\ref{eq:ver4}). For $(m,m,y,t)\in W$, Eq. (\ref{eq:ver4}) and the It\^{o}-Tanaka-Meyer formula imply\begin{small}
\begin{align*}
\E_{m,m,y,t}\!\Biggl[\int_{t}^{(t+h)\wedge\tau_a^{\hat{\mathrm{D}}}}\!\!\bigl[f_m(X_r^{\hat{\mathrm{D}}},M_r^{\hat{\mathrm{D}}},\hat{\mathrm{D}}_r,r,s)
+\beta(r-s)h(M_r^{\hat{\mathrm{D}}},\hat{\mathrm{D}}_r,r)\bigr]\,\mathrm{d}(M^{\hat{\mathrm{D}}})^c_r\Biggr]=0 .
\end{align*}
\end{small}
Assuming $f_m(m,m,y,t,s)+\beta(t-s)h(m,y,t)=\epsilon>0$, we have\begin{small}
\begin{align*}
0&=\E_{m,m,y,t}\Bigg[\int_t^{(t+h)\wedge \tau} f_m(X_r^{\hat{\mathrm{D}}},M_r^{\hat{\mathrm{D}}}, \hat{\mathrm{D}}_r, r,s)+\beta(r-s) h(M_r^{\hat{\mathrm{D}}}, \hat{\mathrm{D}}_r, r)  \mathrm{d}(M^{\hat{\mathrm{D}}})^c_r\Bigg]\\&<\frac\epsilon2\E_{m,m,y,t}\left[\inf_{r\in[t,(t+h)\wedge\tau)}X_r^{\hat{\mathrm{D}}}-m\right],
\end{align*}
\end{small}
where \begin{small}$$\tau:=\tau_a^{\hat{\mathrm{D}}}\wedge \inf_{r\geq t}\left\{f_m(X_r^{\hat{\mathrm{D}}},M_r^{\hat{\mathrm{D}}}, \hat{\mathrm{D}}_r, r,s)+\beta(t-s)h(M_r^{\hat{\mathrm{D}}}, \hat{\mathrm{D}}_r, r)\leq \frac\epsilon2\right\}.$$\end{small}
This is a contradiction to Condition (4). 
Consequently, we obtain $f_m(m,m,y,t,s)+\beta(t-s)h(m,y,t)=0$, and thus Eq. (\ref{eq:ver6}) holds for $(x,m,y,t)\in W$. 
Therefore, Eqs. (\ref{eq:ver2})-(\ref{eq:ver7}) are satisfied.

Then, in Proof of Eq. (\ref{eq:ver1}), we make the same substitution as in the proof of Theorem \ref{th:verify}: we replace $(t+h)-\wedge \tau_a$ with $(t+h)-$. As $V\in C^{2,1,1,1}(\mathcal{Q}_a)$ and $\hat{\mathrm{D}}$ is an equilibrium singular control, we have
\begin{small}
\begin{align} \label{eq:J3}
&J(x,m,y,t;\mathrm{D}^h)-J(x,m,y,t;\hat{\mathrm{D}}) \\[-0.8ex]
&=\E\bigg[ \int_t^{(t+h)\text{-}} \!\! \tfrac{\mathcal{A}V(X_r^{\mathrm{D}^h}\!,M_r^{\mathrm{D}^h}\!,\mathrm{D}_r^h,r) - f_s(X_{(t+h)\text{-}}^{\mathrm{D}^h},M_{(t+h)\text{-}}^{\mathrm{D}^h},\mathrm{D}_{(t+h)\text{-}}^h,t+h,r) + \beta(r-t)H(X_r^{\mathrm{D}^h}\!,M_r^{\mathrm{D}^h}\!,\mathrm{D}_r^h,r)}{h} \d r \notag \\[-0.5ex]
&\quad + \int_t^{(t+h)\text{-}} \!\!\! \big( V_y-V_x+\beta(r-t)c \big)(X_r^{\mathrm{D}^h}\!,M_r^{\mathrm{D}^h}\!,\mathrm{D}_r^h\!,r) \diamond \d \mathrm{D}_r^h \bigg] \le 0.\notag
\end{align}
\end{small}
Consider the perturbation in the form of\begin{small}
\begin{align*}
\mathrm{D}_r^h=
\begin{cases}
y+\Delta y, & r \in [t, t+h), \\[6pt]
\mathrm{D}_r^{X_{(t+h)-}^{\mathrm{D}^h},\ M^{D^h}_{(t+h)-},\ \mathrm{D}_{(t+h)-}^h,\ \hat{\Xi}}, & r \in [t+h, \infty).
\end{cases}
\end{align*}\end{small}
For $\Delta y>0$, taking the limit as $h\to0^+$ in Eq. (\ref{eq:J3}) and using the DCT together with Condition (3), we obtain
\begin{small}
\begin{align*}
\E\Bigg[\int_0^{\Delta y}V_y(x-r,m,y+r,t)-V_x(x-r,m,y+r,t)+c(x-r,m,y+r,t) \d r\Bigg]\leq 0.
\end{align*}
\end{small}
Thus, we have $c(x,m,y,t)-V_x(x,m,y,t)+V_y(x,m,y,t)\leq 0$. Moreover, letting $\Delta y=0$, we have
\begin{small}
\begin{align*}
&\E\Bigg[\frac1h\Bigg(\int_t^{(t+h)-}\mathcal{A}V(X_r^{\mathrm{D}^h},M_r^{\mathrm{D}^h},\mathrm{D}^h_r,r)-f_s(X_{(t+h)-}^{\mathrm{D}^h},M_{(t+h)-}^{\mathrm{D}^h},\mathrm{D}^h_{(t+h)-},t+h,r)
\\ &\quad+\beta(r-t)H(X_r^{\mathrm{D}^h},M_r^{\mathrm{D}^h},\mathrm{D}^h_r,r)\Bigg)\d r\leq 0
\end{align*}
\end{small}
holds for any $h>0$. Hence, we obtain $\mathcal{A}V(x,m,y,t)+H(x,m,y,t)-f_s(x,m,y,t,t)\leq 0$, which yields
\begin{small}$$\mathcal{A}V(x,m,y,t)+H(x,m,y,t)-\mathcal{A}f(x,m,y,t,t)+\mathcal{A}f^t(x,m,y,t)\leq 0.$$\end{small}
Therefore, we obtain Eq. (\ref{eq:ver1}) and the proof is complete.
\end{proof}
Proposition \ref{prop:spec} serves as a supporting case for this theorem; see Remark \ref{remark:nec} below. Condition (4) ensures that the running minimum process is updated and excludes increasing processes such as $X_r=x+r-t,\quad \forall r\geq t$, a.s..
\begin{remark}\label{remark:connect}
In Theorem \ref{th:nec}, we verify the variational inequality through Eq. (\ref{eq:J3}), which requires $V\in C^{2,1,1,1}(\mathcal{Q}_a)$.  Therefore, the second-order smooth fit conditions are typically associated with the establishment of the variational inequality. In contrast, the regularity assumptions on $V$ in Theorem \ref{th:verify} are relatively weaker.

In the next section, we present the existence of the equilibrium through a concrete example to illustrate the interpretation and applicability of Theorems \ref{th:verify} and \ref{th:nec}. The uniqueness of the equilibrium remains an open problem at present.
\end{remark}
\section{A Specific Example}\label{sec:special}
In Subsection \ref{subsec:rigor}, we establish the existence of an equilibrium and verify its key properties, thereby supporting the verification theorem. In Subsection \ref{subsec:num}, we perform numerical simulations to illustrate the shape of the boundary separating the waiting region from the action region. In this section, we abbreviate $X^{\mathrm{D}}$, $M^{\mathrm{D}}$ to $X$ and $M$, respectively.
\subsection{A Dividend Problem under Pseudo-Exponential Discount  Function}\label{subsec:rigor}
Inspired by \citet*{FR2025}, we consider the payoff functional\footnote{In this paper, we focus on the running minimum process. Consequently, we only consider examples where $\mu$ and $\sigma$ are independent of the singular control $\mathrm{D}$.}
$$J(x,m,t;\mathrm{D})=\E_{x,m,t}\left[\int_t^{\tau_{t,a}^{\mathrm{D}}}\beta(r-t)\cdot c (M_r) \diamond \d\mathrm{D}_r\right],$$
where $\beta(t)=\rho\exp\{-\delta t\}+(1-\rho)\exp\{-(\delta+\gamma)t\}$, $\rho \in (0,1), \delta,\gamma>0$, and the definition of $c$ is given by
$$c(m)=
\begin{cases}
e^{-qm}, & m \leq \overline{m}, \\[6pt]
e^{-q\overline{m}}, & m >\overline{m},
\end{cases}
$$
where $\overline{m}\in (0,\infty]$.
The discount function $\beta$ is of the pseudo-exponential form. 
For more details, see, e.g., \citet*{EkelandLazrak2006}, \citet*{K2007}, and \citet*{Bjork2021}. 

The function $c$ is designed to incorporate key insights from dividend smoothing, signaling theory, and the scarring effect. First, firms typically maintain stable dividends to avoid negative signals, cutting them only in severe crises, as supported by \citet*{AMG2013} and \citet*{AAW2015}
. Second, following a major downturn, firms tend to remain conservative for extended periods, consistent with the scarring effect discussed in \citet*{KJL2020} and \citet*{K2021}. This persistent caution motivates modeling $c$ as a function of the running minimum wealth $m$. Consequently, a stable dividend policy is sustained when the historical minimum wealth stays above $\overline{m}$, whereas crossing below $\overline{m}$ triggers a shift to a more conservative strategy.

Assume that $\{X_t\}_{t\geq 0}$ evolves following the SDE $X_t=x+\mu t+\sigma B_t-\mathrm{D}_t,t\geq 0$, where $\mu,\sigma>0$, and assume the bankruptcy boundary $a(m)=0, \quad \forall m\in \R$. Thus, we focus on cases with $m\geq 0$. 

We define $f(x,m,t,s):=\E_{x,m,t}\big[\int_t^{\tau_{t,a}^{\mathrm{D}}}\beta(r-s)c(M_r)\diamond \d\mathrm{D}_r\big]$. 
We consider the case where the free boundary $b$, which serves as the boundary of the waiting region and purchasing region, is independent of $t$. Hence, according to Eq. (\ref{eq:probpre}), we expect $f$ to be related only to $x,m$ and $\kappa=t-s$. This is similar to the time-homogeneous problem in traditional singular control problems. Then, 
we only need to study the following free boundary problem according to Theorem \ref{th:verify}:
\begin{small}
\begin{align}
&f_\kappa(x,m,t-s)+\mu f_x(x,m,t-s)+\tfrac{1}{2}\sigma^2f_{xx}(x,m,t-s)=0, \ 0\le m\le x< b(m), \label{eq:special1} \\[-0.8ex]
&f_x(x,m,t-s)=\beta(t-s)c(m), \ b(m)\le x, \label{eq:special2} \\[-0.8ex]
&f_\kappa(x,m,0)+\mu f_x(x,m,0)+\tfrac{1}{2}\sigma^2f_{xx}(x,m,0)\le 0, \ 0\le m \le x, \label{eq:special3} \\[-0.8ex]
&f_x(x,m,0)\ge e^{-qm}, \ 0\le m \le x, \label{eq:special4} \\[-0.8ex]
&f_m(x,m,t-s)|_{x=m}=0, \ m\ge 0, \label{eq:special5} \\[-0.8ex]
&f(0,0,t-s)=0, \quad V(x,m)=f(x,m,0). \label{eq:special6}
\end{align}
\end{small}
Motivated by Eqs. (\ref{eq:special1}) and (\ref{eq:special2}), and the definition of $\beta$ and $c$, we expect that the solution of the above PDE has the form of separation of variables\begin{small}
$$f(x,m,\kappa)=\sum_{i=1}^4A_i(m)e^{\lambda_ix-(\mu\lambda_i+\frac{\sigma^2}{2}\lambda_i^2)\kappa}$$\end{small}
for $0\leq m\leq \overline{m}$ and $m\leq x\leq b(m)$, where $\lambda_1$ and $\lambda_2$ are solutions of the equation $\frac{\sigma^2}{2}\lambda^2+\mu \lambda= \delta$ while $\lambda_3$ and $\lambda_4$ are solutions of the equation $\frac{\sigma^2}{2}\lambda^2+\mu \lambda= \delta+\gamma$. We assume $\lambda_3<\lambda_1<0<\lambda_2<\lambda_4$ from now on. From Eqs. (\ref{eq:special5}) and (\ref{eq:special6}), we deduce
\begin{small}
\begin{align}
&A_1'(m)e^{\lambda_1 m}+A_2'(m)e^{\lambda_2 m}=0, \label{eq:special51}\\
&A_3'(m)e^{\lambda_3 m}+A_4'(m)e^{\lambda_4 m}=0, \label{eq:special52}\\
&A_1(0)+A_2(0)=0, \label{eq:special53}\\
&A_3(0)+A_4(0)=0. \label{eq:special54}
\end{align}
\end{small}
In view of Condition (1) in Theorem \ref{th:verify}, we require $V\in C^{2,1}(\{(x,m)|x\geq m,0\leq m\leq \overline{m}\})$ and $f\in C^{1,1,1}(\{(x,m,\kappa)|0\leq m\leq x,0\leq m\leq \overline{m},\kappa \geq 0\})$. We first consider the smooth fit conditions regarding $x$ on the free boundary $b$. According to Theorem \ref{th:verify} and Remark \ref{remark:connect}, we require
\begin{small}
\begin{align}
&f_x(x,m,\kappa)|_{x=b(m)^+} = f_x(x,m,\kappa)|_{x=b(m)^-}, \label{eq:special21} \\[-0.8ex]
&V_{xx}(x,m)|_{b(m)^+} = f_{xx}(x,m,0)|_{b(m)^+} = f_{xx}(x,m,0)|_{b(m)^-} = V_{xx}(x,m)|_{b(m)^-}\label{eq:special22}
\end{align}
\end{small}
for any $0\leq m\leq \overline{m}$.
Eqs. (\ref{eq:special21}) and (\ref{eq:special22}) are equivalent to
\begin{small}
\begin{align}
&\lambda_1A_1(m)e^{\lambda_1b(m)} + \lambda_2A_2(m)e^{\lambda_2 b(m)} = \rho e^{-qm}, \label{eq:spec21} \\[-0.8ex]
&\lambda_3A_3(m)e^{\lambda_3b(m)} + \lambda_4A_4(m)e^{\lambda_4 b(m)} = (1-\rho)e^{-qm}, \label{eq:spec22} \\[-0.8ex]
&\textstyle \sum_{i=1}^2 \lambda_i^2 A_i(m)e^{\lambda_i b(m)} + \sum_{j=3}^4 \lambda_j^2 A_j(m)e^{\lambda_j b(m)} = 0. \label{eq:spec23}
\end{align}
\end{small}
We find Eqs. (\ref{eq:special51})-(\ref{eq:special54}) and (\ref{eq:spec21})-(\ref{eq:spec23}) form a differential-algebraic equation system; see, e.g.,  \citet*{2006DAE}. We are able to transform it into an ODE system in the following way.

Let $F(m)=\lambda_1^2A_1(m)e^{\lambda_1b(m)}+\lambda_2^2A_2(m)e^{\lambda_2b(m)}$. Using Eq. (\ref{eq:spec21}), we have\begin{small}
\begin{align}\label{eq:spec31}
A_1(m)=\frac{1}{\lambda_1(\lambda_2-\lambda_1)}e^{-\lambda_1 b(m)}(\rho e^{-qm}\lambda_2-F(m)),
\end{align}
\end{small}
\begin{small}
\begin{align}\label{eq:spec32}
A_2(m)=\frac{1}{\lambda_2(\lambda_1-\lambda_2)}e^{-\lambda_2 b(m)}(\rho e^{-qm}\lambda_1-F(m)).
\end{align}
\end{small}
Similarly, we obtain\begin{small}
\begin{align}\label{eq:spec41}
A_3(m)=\frac{1}{\lambda_3(\lambda_4-\lambda_3)}e^{-\lambda_3 b(m)}((1-\rho) e^{-qm}\lambda_4+F(m)),
\end{align}\end{small}
\begin{small}
\begin{align}\label{eq:spec42}
A_4(m)=\frac{1}{\lambda_4(\lambda_3-\lambda_4)}e^{-\lambda_4 b(m)}((1-\rho) e^{-qm}\lambda_3+F(m)).
\end{align}\end{small}
Substituting Eqs. (\ref{eq:spec31}) and (\ref{eq:spec32}) into Eq. (\ref{eq:special53}) yields\begin{small}
\begin{align}\label{eq:spec33}
&F(0)=\rho\frac{\lambda_2^2e^{-\lambda_1b(0)}-\lambda_1^2e^{-\lambda_2b(0)}}{\lambda_2e^{-\lambda_1b(0)}-\lambda_1e^{-\lambda_2b(0)}}. 
\end{align}\end{small}
Analogously, we have\begin{small}
\begin{align}\label{eq:spec34}
&F(0)=-(1-\rho)\frac{\lambda_4^2e^{-\lambda_3b(0)}-\lambda_3^2e^{-\lambda_4b(0)}}{\lambda_4e^{-\lambda_3b(0)}-\lambda_3e^{-\lambda_4b(0)}}.
\end{align}\end{small}
Substituting Eqs. (\ref{eq:spec31})-(\ref{eq:spec42}) into Eqs. (\ref{eq:special51}) and (\ref{eq:special52}) yields\begin{small}
\begin{align}\label{eq:ODEFb}
X'
\begin{bmatrix}
b'(m) \\ F'(m)
\end{bmatrix}
=
\begin{bmatrix}
X_{11}(m) & X_{12}(m)  \\
X_{21}(m) &  X_{22}(m)
\end{bmatrix}
\begin{bmatrix}
b'(m) \\ F'(m)
\end{bmatrix}
=
\begin{bmatrix}
R_1(m) \\
R_2(m)
\end{bmatrix},
\end{align}\end{small}
where 
\begin{small}
\begin{align} \label{eq:XR}
&X_{11}(m) = \lambda_1\lambda_2 \big[ \textstyle \rho e^{-qm} (\lambda_1 e^{\lambda_2(m-b(m))} \!-\! \lambda_2 e^{\lambda_1(m-b(m))}) + (e^{\lambda_1(m-b(m))} \!-\! e^{\lambda_2(m-b(m))}) F \big],  \\[-0.0ex]
&X_{12}(m) = \lambda_1 e^{\lambda_2(m-b(m))} - \lambda_2 e^{\lambda_1(m-b(m))}, \notag \\[-0.0ex]
&X_{21}(m) \!\!=\!\! \lambda_3\lambda_4 \big[ \textstyle (1\!-\!\rho)e^{-qm} (\lambda_3 e^{\lambda_4(m-b(m))} -\! \lambda_4 e^{\lambda_3(m-b(m))}) - (e^{\lambda_3(m-b(m))} \!-\! e^{\lambda_4(m-b(m))}) F \big], \notag \\[-0.0ex]
&X_{22}(m) = \lambda_4 e^{\lambda_3(m-b(m))} - \lambda_3 e^{\lambda_4(m-b(m))}, \notag \\[-0.0ex]
&R_1(m) = \rho q e^{-qm} (\lambda_2^2 e^{\lambda_1(m-b(m))} - \lambda_1^2 e^{\lambda_2(m-b(m))}), \notag \\[-0.0ex]
&R_2(m) = (1-\rho) q e^{-qm} (\lambda_4^2 e^{\lambda_3(m-b(m))} - \lambda_3^2 e^{\lambda_4(m-b(m))}).\notag
\end{align}
\end{small}
Next, we prove the existence and uniqueness of a local solution for this ODE system and utilize it to construct an equilibrium.

Eqs. (\ref{eq:spec33}) and (\ref{eq:spec34}) indicate that $b(0)$ is a solution of the equation $G(b)=0$, where $G$ is\begin{small}
\begin{align}\label{eq:G}
G(b)=\rho\frac{\lambda_2^2e^{-\lambda_1b}-\lambda_1^2e^{-\lambda_2b}}{\lambda_2e^{-\lambda_1b}-\lambda_1e^{-\lambda_2b}}+(1-\rho)\frac{\lambda_4^2e^{-\lambda_3b}-\lambda_3^2e^{-\lambda_4b}}{\lambda_4e^{-\lambda_3b}-\lambda_3e^{-\lambda_4b}}.
\end{align}\end{small}
The function $G$ is strictly monotonic as $G'(b)<0$ for any $b\in \R$. By the Intermediate Value Theorem, it follows that \begin{small}$b(0)\in\left(\frac{\ln(\frac{\lambda_3}{\lambda_4})^2}{\lambda_4-\lambda_3},\frac{\ln(\frac{\lambda_1}{\lambda_2})^2}{\lambda_2-\lambda_1}\right)$\end{small}. Consequently, we obtain $b(0)>0$ and hence $F(0)<0$.

By direct calculation, we have $X_{11}(0)>0$, $X_{12}(0)<0$, and $X_{22}(0)>0$. Besides,
\begin{align}\label{eq:F(0)}
F(0)+(1-\rho)\lambda_4>0
\end{align}
holds because of Eq. (\ref{eq:spec34}) and the fact that \begin{small}$b\mapsto \frac{\lambda_4^2e^{-\lambda_3b}-\lambda_3^2e^{-\lambda_4b}}{\lambda_4e^{-\lambda_3b}-\lambda_3e^{-\lambda_4b}}$\end{small} is increasing with respect to $b$ on $[0,\infty)$. Hence, we have \begin{small}$$X_{21}(0)>-(\lambda_3\lambda_4)[(1-\rho)\lambda_4+F(0)](e^{-\lambda_3b(0)}-e^{-\lambda_4b(0)})>0,$$\end{small}
and thus $\det(X)=X_{11}X_{22}-X_{21}X_{12}>0$.

Due to the continuity of $b$ and $F$, we have $\det(X)$, which can be viewed as a function of  $(m,b,F)$, has a common lower bound $C_X>0$ in some region $R_{m,b,F}:=[0,\epsilon_1]\times[b(0)-h_1,b(0)+h_1]\times [F(0)-h_2,F(0)+h_2]$, where $\epsilon_1,h_1,h_2>0$ are constants. We are able to choose small enough $\epsilon_1>0$ and $h_1>0$ satisfying $b(0)-h_1>\epsilon_1$. As $\det(X)$ has a lower bound, we can easily verify that each element of the matrix $Y=X^{-1}$ fulfills the locally Lipschitz condition with respect to $b$ and $F$ in $R_{m,b,F}$. Hence, Eqs. (\ref{eq:spec33})-(\ref{eq:XR}) admit a unique solution in some closed region $R'_{m,b,F}=[0,\epsilon_2]\times [b(0)-h_3,b(0)+h_3]\times[F(0)-h_4,F(0)+h_4]\subset R_{m,b,F}$, where $\epsilon_2,h_3,h_4>0$. Based on these results, we then prove the existence of the equilibrium under certain parameter conditions.
\begin{proposition}\label{prop:spec}
Assume $0<\overline{m}<\epsilon_2$ and \begin{small}$$0\leq \rho< \frac{-\mu +\sqrt{\mu^2+2\sigma^2(\delta+\gamma)}}{\sqrt{\mu^2+2\sigma^2\delta}+\sqrt{\mu^2+2\sigma^2(\delta+\gamma)}},$$\end{small}
then Eqs. (\ref{eq:spec33})-(\ref{eq:XR}) admit a unique solution $(b,F)$ for $m\in [0,\overline{m}]$. Moreover, for this same $b$, Eqs. (\ref{eq:special51})-(\ref{eq:special54}) and (\ref{eq:spec21})-(\ref{eq:spec23}) also admit a unique solution $(A_1,A_2,A_3,A_4,b)$ on $m\in [0,\overline{m}]$.
Define $m^*=b(\overline{m})$, and for $m> \overline{m}$, define $b(m)=b(\overline{m})$, $F(m)=F(\overline{m})$, and $A_i(m)=A_i(\overline{m})$. We have
\begin{small}
\begin{align}
V(x,m)=\left\{ \begin{aligned}
&\textstyle \sum_{i=1}^4 A_i(m)e^{\lambda_i x}, && \text{\scriptsize $0\le m\le \overline{m}, m\le x\le b(m),$} \\
&\textstyle \sum_{i=1}^4 A_i(\overline{m})e^{\lambda_i x}, && \text{\scriptsize $\overline{m}< m\le m^*, m\le x \le b(m),$} \\
&(x\!-\!b(m))e^{-qm} + \textstyle \sum_{i=1}^4 A_i(m)e^{\lambda_i b(m)}, && \text{\scriptsize $0\le m\le \overline{m}, b(m)<x,$} \\
&(x\!-\!b(\overline{m}))e^{-q\overline{m}} + \textstyle \sum_{i=1}^4 A_i(\overline{m})e^{\lambda_i b(\overline{m})}, && \text{\scriptsize $\overline{m}< m, b(m)\vee m\le x,$}
\end{aligned} \right.
\end{align}
\end{small}
is an equilibrium satisfying the conditions of Definition \ref{def:equi}. The corresponding $f$ is given by\begin{small}
$$f(x,m,\kappa)=
\begin{cases}
f_1(x,m,\kappa), &0\leq m\leq \overline{m},m\leq x\leq b(m), \\[6pt]
f_2(x,m,\kappa), & \overline{m}< m\leq m^*,m\leq x \leq b(m), \\[6pt]
f_3(x,m,\kappa), & 0\leq m\leq \overline{m},b(m)<x, \\[6pt]
f_4(x,m,\kappa), &\overline{m}< m,b(m)<x ,
\end{cases}
$$\end{small}
where \begin{small}
\begin{align*}
f_1(x,m,\kappa) &= \textstyle \sum_{i=1}^4 A_i(m) e^{\lambda_i x - (\mu\lambda_i + \frac{1}{2}\sigma^2\lambda_i^2)\kappa}, \
f_2(x,m,\kappa) = \textstyle \sum_{i=1}^4 A_i(\overline{m}) e^{\lambda_i x - (\mu\lambda_i + \frac{1}{2}\sigma^2\lambda_i^2)\kappa}, \\[-0.8ex]
f_3(x,m,\kappa) &= (x\!-\!b(m))e^{-qm} \big( \rho e^{-\delta\kappa} \!+\! (1\!-\!\rho)e^{-(\delta+\gamma)\kappa} \big) + \textstyle \sum_{i=1}^4 A_i(m) e^{\lambda_i b(m) - (\mu\lambda_i + \frac{1}{2}\sigma^2\lambda_i^2)\kappa}, \\[-0.8ex]
f_4(x,m,\kappa) &\!\!=\! (x\!-\!b(\overline{m}))e^{-q\overline{m}} \big( \rho e^{-\delta\kappa} \!+\! (1\!-\!\rho)e^{-(\delta+\gamma)\kappa} \big) \!\! +\!\! \textstyle \sum_{i=1}^4 A_i(\overline{m}) e^{\lambda_i b(\overline{m}) - (\mu\lambda_i + \frac{1}{2}\sigma^2\lambda_i^2)\kappa}.
\end{align*}
\end{small}
\end{proposition}
\begin{proof}
We verify that each condition of Proposition \ref{prop:f^s} holds to demonstrate that $V$ is an equilibrium:

\textbf{Part I. The equation part of the PDE and Condition (1) of Theorem \ref{th:verify}.} First, $\overline{m}<m^*$ holds\footnote{For the case $\overline{m}\geq m^*$, there exists a similar proof, but the forms of $V$ and $f$ are different.} because we have let $b(0)-h_1>\epsilon_1$. The equivalence of system Eqs.  (\ref{eq:spec33})-(\ref{eq:XR}) and the system Eqs. (\ref{eq:special51})-(\ref{eq:special54}) and (\ref{eq:spec21})-(\ref{eq:spec23}) is easily verified through Eqs. (\ref{eq:spec31})-(\ref{eq:spec34}).

According to the definitions of $V$ and $f$, we can verify Eqs. (\ref{eq:special1}), (\ref{eq:special2}), (\ref{eq:special5}), and (\ref{eq:special6}) by direct calculation. Additionally, Eqs. (\ref{eq:special21}) and (\ref{eq:special22}) hold because $b(m)=b(\overline{m}),F(m)=F(\overline{m})$ when $m\geq \overline{m}$.

We then verify Condition (1) in Theorem \ref{th:verify}. Calculation shows that Eqs. (\ref{eq:special51})-(\ref{eq:special54}) and (\ref{eq:spec21})-(\ref{eq:spec23}) hold on $(\overline{m},\infty)$. The required regularity of $V_x$, $V_{xx}$ and $f_x$ is then guaranteed by these equations. Then, we only need to verify the regularity conditions of $m$ and $\kappa$. For $f_\kappa$, 
we have
\begin{small}
\begin{align*}
f_\kappa(x,m,\kappa)|_{b(m)^+} &= -\textstyle \sum_{i=1}^4 \! \big(\mu\lambda_i \!+\! \frac{1}{2}\sigma^2\lambda_i^2\big) A_i(m) e^{\lambda_i x - (\mu\lambda_i + \frac{1}{2}\sigma^2\lambda_i^2)\kappa} \\[-0.6ex]
&= f_\kappa(x,m,\kappa)|_{b(m)^-}
\end{align*}
\end{small}
holds, and we can similarly prove the cases where $m>\overline{m}$. The regularity conditions for $V_m$ and $f_m$ on $[0,\overline{m})\cup (\overline{m},\infty)$ are obviously satisfied (noting the smooth fit condition on $b$ is not necessary for $V_m$ and $f_m$). Hence, we have $V\in C^{2,1}([0,\infty)\times[0,\overline{m}])\cap C^{2,1}([0,\infty)\times[\overline{m},\infty))$ and $f\in C^{2,1,1}([0,\infty)\times[0,\overline{m}]\times[0,\infty))\cap C^{2,1,1}([0,\infty)\times[\overline{m},\infty)\times[0,\infty))$. Therefore, in Theorem \ref{th:verify}, all instances where the It\^{o}-Meyer-Tanaka formula is used can be replaced by using Theorem 3.1 in \citet*{PG2005} to obtain the same results ( Informally, one can verify $x\mapsto V(x,x)$ and $x\mapsto f(x,x,\kappa)$ are $C^1$, and $M$ only varies when $X=M$. Thus, the It\^{o}-Tanaka-Meyer formula holds\footnote{The reason why $V$ loses the $C^1$ regularity is that $c$ is not $C^1$. This can be interpreted as the company suddenly changing its dividend strategy when facing a major crisis according to the dividend smoothing theory and dividend signaling theory; e.g., see \citet*{K2021}.}).

\textbf{Part II.  Conditions (4) and (5) of Theorem \ref{th:verify}.}
We prove the Skorokhod reflection problem with boundary $b$ admits a unique solution. Consider the Skorokhod problem defined by the equation\begin{small}
\begin{align}
d_\kappa = x_\kappa - \sup_{t \le r \le \kappa} (x_r - b(M_r^d))^+, \quad M_\kappa^d = m\wedge \inf_{t \le r \le \kappa} d_r.
\end{align}\end{small}
Define the following closed regions\begin{small}
\begin{align*}
& S_1 = \{ (x, m) \mid 0 \le m \le \overline{m}, \, m \le x \le \frac{m + 2b(m)}{3} \},\\
&S_2 = \{ (x, m) \mid 0 \le m \le \overline{m}, \, \frac{m + 2b(m)}{3} \le x \le \frac{2 m + b(m)}{3} \}.
\end{align*}\end{small}
For a fixed $m \in (\overline{m}, m^*)$, define a sequence of stopping times $\{\tau_n\}_{n \in \mathbb{N}}$. Let $\tau_1 = \inf \{ r \ge t \mid (d_r, M_r^d) \in S_1 \}$. For $n \ge 1$, define recursively\begin{small}
\begin{align*}
&\tau_{2n} = \inf \{ r \ge \tau_{2n-1} \mid (d_r, M_r^d) \in S_2 \}, \\
&\tau_{2n+1} = \inf \{ r \ge \tau_{2n} \mid (d_r, M_r^d) \in S_1 \}.
\end{align*}\end{small}
We construct the solution $(d, k)$ as follows.\\
1. On $[t, \tau_1]$, let $(d, k)$ be the solution to the standard Skorokhod problem with a fixed boundary $b(m^*)$.\\
2. On $[\tau_{2n-1}, \tau_{2n}]$, the process evolves following $\{x_r\}_{r\geq t}$: $d_r = d_{\tau_{2n-1}} + x_r - x_{\tau_{2n-1}}$ and $k_r = k_{\tau_{2n-1}}$.\\
3. On $[\tau_{2n}, \tau_{2n+1}]$, let $(d, k)$ be the solution to the Skorokhod problem with a fixed boundary frozen at $b(M_{\tau_{2n}}^d)$.

By the existence of solutions for fixed-boundary Skorokhod problems; see, e.g. \citet*{KS1991}, the concatenation of these paths yields a continuous solution for the moving boundary problem.

Assume that there exist two distinct solutions $d$ and $d'$ for the same driving process $x_t$ and initial condition $x_0$. Let $t_1 = \inf \{ r \ge 0 \mid d_r \neq d'_r \}$ be the first time of divergence.
By continuity, we have $d_{t_1} = d'_{t_1}$ and $M_{t_1}^d = M_{t_1}^{d'}$. For some $t_2 > t_1$, the paths $(d_r, M_r^d)$ and $(d'_r, M_r^{d'})$ remain in region $S_1$ for all $r \in [t_1, t_2]$.

In $S_1$, the reflection term $\sup (x_r - b(M_r^d))^+$ is locally constant. Specifically, the reflection term $k_t$ only increases when the process hits the boundary. As the paths are within $S_1$, the reflection equation reduces to $d_T = x_T + C$,
where $C$ is a constant determined by the state at $t_1$. As $d_{t_1} = d'_{t_1}$ and both follow the same increment $x_T - x_{t_1}$, it follows that $d_r = d'_r$ on $[t_1, t_2]$. This contradicts the definition of $t_1$ as the divergence time. Thus, the solution is unique. Consequently, for any $\omega\in \Omega$ and the continuous input $x_r=\mu t+\sigma B_t(\omega)$, using the above results, we obtain that the Skorokhod reflection problem (\ref{eq:Skorokhod}) 
has a unique strong solution, denoted by $(\{X_{r}\}_{r\geq t},\{M_{r}\}_{r\geq t},\{\hat{\mathrm{D}}_r\}_{r\geq t})$.

Then, using Proposition \ref{prop:Skorokhod} (in the proof, for the integrability condition to hold, $L_m<1$ is unnecessary) yields \begin{small}
$
\E_{x,m,y,t}\Big[\int_0^{\tau_a^{\hat{\mathrm{D}}}}\left|\beta(r-s)c(X_r^{{\hat{\mathrm{D}}}},M_r^{\hat{{\mathrm{D}}}},\hat{\mathrm{D}}_r,r)\right|\diamond\d \hat{\mathrm{D}}_r\Big]<\infty $. \end{small}
Additionally, following the proof of Proposition \ref{prop:f^s}, we obtain Eq. (\ref{eq:Dforstep3}).

The fact $\lim\limits_{s\to t}f(x,m,t-s)=V(x,m)$ enables us to prove an alternative to Eqs. (\ref{eq:dominate}) and (\ref{eq:dominate2}) to realize \textbf{Step 2} and \textbf{Step 3} in the proof of Theorem \ref{th:verify}, as stated in Remark \ref{remark:int}. Specifically, 
as $b$ is bounded, we have $X_t^{\hat{\mathrm{D}}}$ is bounded for any $t\geq s$. Because $f$ and $b$ are continuous and $\{(x,m)|0\leq m \leq m^*,m\leq x\leq b(m)\}$ is bounded, we obtain $|f(x,m,\kappa)|\leq C(1+x)e^{-\delta \kappa}, \quad \forall 0\leq m\leq x$ holds for some $C>0$. Thus, the boundedness of $X_r^{\hat{\mathrm{D}}}$ yields $\mathbb{E}_{x,m,y,t}\Big[ \sup_{r \in [t, \tau_a^{\hat{\mathrm{D}}}]} \left| f^s(X_r^{\hat{\mathrm{D}}}, M_r^{\hat{\mathrm{D}}},r) \right|\Big] < \infty$ for any $s<t$. Then, we have \begin{small}$$\limsup_{T \to \infty} \mathbb{E}_{x,m,y,t}\left[ \left| f^s(X_T^{\hat{\mathrm{D}}}, M_T^{\hat{\mathrm{D}}},  T) \right| \cdot\mathbf{1}_{\{T < \tau_a^{\hat{\mathrm{D}}}\}} \right] = 0$$\end{small} holds for $s< t$ using the DCT. Hence, we obtain Eq. (\ref{eq:probpre}) for $s<t$ following the proof of \textbf{Step 1}. Then, according to Remark \ref{remark:int}, we prove Condition (5) of Theorem \ref{th:verify}.

\textbf{Part III. Conditions (1)-(3) of Proposition \ref{prop:f^s} and the variational inequality part of the PDE.} Because  Conditions (1)-(3) in Proposition \ref{prop:f^s} are straightforward to verify, we only need to verify Eq. (\ref{eq:ver1}), or equivalently, Eqs. (\ref{eq:special3}) and (\ref{eq:special4}). Then, all the assumptions (or equivalent conditions) of Proposition \ref{prop:f^s} are satisfied, and we have $W$ given by $W=\{(x,m,t)| 0\leq m \leq m^*,m\leq x\leq b(m),t\in [0,\infty)\}$, $\{\hat{\mathrm{D}}_r\}_{r\geq t}$ is the equilibrium singular control generated by $P=\{(x,m)|0\leq m\leq x \leq b(m)\}$, while $V$ is the equilibrium payoff functional.

In fact,  Eq. (\ref{eq:special3}) follows from straightforward calculation and repeated use of Eqs. (\ref{eq:spec21})-(\ref{eq:spec23}). Specifically, we have
\begin{small}
\begin{align*}
&(f_\kappa + \mu f_x + \tfrac12 \sigma^2 f_{xx})(x,m,0) \\[-1.2ex]
&\quad= -e^{-qm}(x-b(m))\bigl(\rho\delta+(1-\rho)(\delta+\gamma)\bigr) - \mu\sum_{i=1}^4 \lambda_i A_i(m) e^{\lambda_i b(m)} \\[-1.4ex]
&\qquad - \frac{\sigma^2}{2}\sum_{i=1}^4 \lambda_i^2 A_i(m) e^{\lambda_i b(m)} + \mu e^{-qm} \\[-1.2ex]
&\quad= -e^{-qm}(x-b(m))\bigl(\rho\delta+(1-\rho)(\delta+\gamma)\bigr) \leq 0, \qquad \forall m \in [0, \overline{m}], \\[0.5ex]
&(f_\kappa + \mu f_x + \tfrac12 \sigma^2 f_{xx})(x,m,0) \\[-1.2ex]
&\quad= -e^{-q\overline{m}}(x-b(\overline{m}))\bigl(\rho\delta+(1-\rho)(\delta+\gamma)\bigr) - \mu\sum_{i=1}^4 \lambda_i A_i(\overline{m}) e^{\lambda_i b(\overline{m})} \\[-1.4ex]
&\qquad - \frac{\sigma^2}{2}\sum_{i=1}^4 \lambda_i^2 A_i(\overline{m}) e^{\lambda_i b(\overline{m})} + \mu e^{-q\overline{m}} \\[-1.2ex]
&\quad= -e^{-q\overline{m}}(x-b(\overline{m}))\bigl(\rho\delta+(1-\rho)(\delta+\gamma)\bigr) \leq 0, \qquad \forall m \in (\overline{m}, \infty).
\end{align*}
\end{small}
To verify Eq. (\ref{eq:special4}), we first show that $A_3(m)\le 0$ for all $m\ge 0$. Suppose, for contradiction, that $A_3(m)\ge 0$ for some $m\ge 0$. Because Eqs. (\ref{eq:special54}) and (\ref{eq:spec22}) give $A_3(0)<0$, the Intermediate Value Theorem yields a point $\overline{m}\ge 0$ such that $A_3(\overline{m})=0$. Substituting $A_3(\overline{m})=0$ into Eqs. (\ref{eq:spec21})-(\ref{eq:spec23}) and using the fact that $A_2(\overline{m})\geq 0$ we obtain an inequality $\rho\lambda_1+(1-\rho)\lambda_4\leq0$, however, this is a contradiction to\begin{small} $$0\leq \rho< \frac{-\mu +\sqrt{\mu^2+2\sigma^2(\delta+\gamma)}}{\sqrt{\mu^2+2\sigma^2\delta}+\sqrt{\mu^2+2\sigma^2(\delta+\gamma)}},$$\end{small} which implies $\rho \lambda_1+(1-\rho)\lambda_4 > 0$. Hence, we conclude that $A_3(m) \le 0$ for all $m \in [0, \overline{m}]$, and consequently for every $m \ge 0$.

According to Eqs. (\ref{eq:spec21})-(\ref{eq:spec23}), as $\lambda_1,\lambda_3<0$ and $\lambda_2,\lambda_4>0$, for a fixed $m\geq 0$, we have $A_2(m)\geq 0$, and $A_4(m)\geq 0$. Besides, either $A_1(m)\leq 0$ or $A_3(m)\leq 0$ or both $A_1(m)>0$ and $A_3(m)\leq 0$ holds. 
Define \begin{small}$$U(x,m):=\sum_{i=1}^4\lambda_iA_i(m)e^{\lambda_ix}-e^{-qm}.$$\end{small} Fix $m\in[0,\overline{m}]$. On one hand, if $A_1(m)\leq 0$, the function $U_x$ is obviously increasing on $[0,b(m)]$ with respect to $x$. We thus obtain $U_x(x,m)\leq U_x(b(m),m)=0$. Hence, $U(x,m)\geq U(b(m),m)=0$ for any $m\leq x\leq b(m)$. On the other hand, if $A_1(m)>0$, because $\lambda_3<\lambda_1<0<\lambda_2<\lambda_4$, we have $\hat{U}(x,m)=e^{-\lambda_3x}U_x(x,m)$ is increasing with respect to $x$, indicating $U_x(x,m)\leq e^{-\lambda_3 (b(m)-x)}U_x(b(m),m)=0$ and $U(x,m)\geq U(b(m),m)=0$ for any $0\leq m\leq m^*,m\leq x\leq b(m)$. Thus, Eq. (\ref{eq:special4}) holds. Therefore, we have proved Eqs. (\ref{eq:special3}) and (\ref{eq:special4}).

Combining the results of the three parts, we obtain all the premises of Proposition \ref{prop:f^s}, namely, Conditions (1), (3), (4), and (5) of Theorem \ref{th:verify}, and Conditions (1) and (2) of Proposition \ref{prop:f^s}. Therefore, we complete the proof.
\end{proof}
\begin{remark}
According to \citet*{AMG2013}, assuming $0<\overline{m}<\epsilon_2$ is economically valid since firms rarely alter dividend policies before severe crises, implying a small $\overline{m}$. However, the restrictions on $\overline{m}$ and $\rho$ in Proposition \ref{prop:spec} can potentially be relaxed. For $\overline{m}>\epsilon_2$ (including $\overline{m}=\infty$), if the $2\times 2$ matrix $X$ in Eq. (\ref{eq:ODEFb}) is invertible, Picard's theorem guarantees the existence and uniqueness of the solution $(b,F)$ with a maximal existence interval $[0,\infty)$ for $b$. Nevertheless, directly verifying the invertibility of $X$ is challenging due to the highly nonlinear, coupled structure of the equations. Furthermore, the condition on $\rho$ is merely sufficient for the associated variational inequality and may not be necessary. These observations are numerically illustrated in Subsection \ref{subsec:num}.

Through numerical experiments, \citet*{CZ2025} report that under certain parameter configurations, the required variational inequality fails, leading to the non-existence of a time-independent equilibrium. This phenomenon is caused by the nonlinear expectations inherent in the mean-variance problem. In contrast, for our non-exponential discounting example, extensive numerical testing (covering $|\lambda_i|<10\,(i=1,2,3,4)$, $0<\rho<1$, and $0<q<5$) has not yielded any instance where an equilibrium fails to exist.
\end{remark}
\begin{remark}
When $\rho=0$ or $\rho=1$, the framework reduces to a time-consistent singular control model. In this case, the equilibrium payoff functional $V$ coincides with the standard optimal value function, and Eq. (\ref{eq:ver1}) can be verified to hold by direct calculation. For further details, we refer to \citet*{FR2025}.

If $q=0$, the model does not depend on the running minimum process. A barrier solution exists with $b\equiv x^*$, and the corresponding coefficients $A_i$ $(i=1,2,3,4)$ become constants. Moreover, $V$ remains an equilibrium payoff functional as $A_3=A_3(0)<0$. Specifically, the unique solution $b^*$ of the equation
\begin{small}$$\rho\frac{\lambda_2^2e^{-\lambda_1b}-\lambda_1^2e^{-\lambda_2b}}{\lambda_2e^{-\lambda_1b}-\lambda_1e^{-\lambda_2b}}+(1-\rho)\frac{\lambda_4^2e^{-\lambda_3b}-\lambda_3^2e^{-\lambda_4b}}{\lambda_4e^{-\lambda_3b}-\lambda_3e^{-\lambda_4b}}=0$$\end{small}
exists, and the free boundary is given by $b=b^*$.
In this case, the model reduces to the pseudo-exponential problem without the running minimum process. Therefore, our example also supports the verification theorem for the time-inconsistent singular control problem; e.g., see \citet*{LLY2024}.

Moreover, when $\rho=q=0$, the equilibrium payoff in Proposition \ref{prop:spec} coincides with the value function for the classical one-dimensional (De Finetti) optimal dividend problem.
\end{remark}
\begin{remark}\label{remark:nec}
We point out that Proposition \ref{prop:spec} satisfies all the conditions of Theorem \ref{th:nec}. This example does not satisfy Condition (1) with respect to $m$ only when $m=\overline{m}$ and $x>\overline{m}$, which does not affect the proof and the result of the theorem. Condition (2) is verified by imitating the proof of Proposition \ref{prop:f^s} (and, Eqs. (\ref{eq:ver1})-(\ref{eq:ver7}) do not require the existence of $V_m$ and $f_m$ for $x>m$). The uncontrolled drift-diffusion dynamics $X_t=\mu \d t+\sigma\d B_t$ in this section satisfy Condition (4). In fact, because \begin{small}$\E\left[\inf_{0\leq r\leq t}\left\{B_r\right\}-B_0\right]=-\sqrt{\frac{2t}{\pi}}$\end{small}, we have \begin{small}$\E_{m,m,y,t}\left[\inf_{r\in[t,t_1)}X_r^{\hat{\mathrm{D}}}\right]<m$\end{small} for any $t_1\geq t$, and because of the path continuity of one-dimensional Brownian motion starting from 0 and the fact that it passes through 0 countless times within any short period of time, we can verify Condition (4) easily by contradiction.
\end{remark}
Then, we investigate properties of the free boundary $b$.
\begin{proposition}\label{prop:b}
The free boundary $b$ in Proposition \ref{prop:spec} satisfies $b\in C^{2}([0,\overline{m})\cup (\overline{m},\infty))\cap C([0,\infty))$, $b'(0)=0$, $b$ is decreasing on $[0,\infty)$, and $b$ is concave in a neighborhood of $0$.
\end{proposition}
\begin{proof}
According to Eqs. (\ref{eq:spec33}) and (\ref{eq:spec34}), we have $\frac{R_1(0)}{X_{12}(0)}=\frac{R_2(0)}{X_{22}(0)}=-qF(0)$. Substituting it into Eq. (\ref{eq:ODEFb}) yields\begin{small} $$\frac{X_{11}(0)}{X_{12}(0)}b'(0)+F'(m)=\frac{X_{21}(0)}{X_{22}(0)}b'(0)+F'(m)=-qF(0).$$\end{small}
As $\det(X(0))\neq 0$, we have $b'(0)=0$ and $F'(0)=-qF(0)$.

Using Eq. (\ref{eq:ODEFb}) and the fact that $\det(X(m))$ has a uniform lower bound $C>0$ on $[0,\epsilon_2]$, we have\begin{small}
\begin{align}\label{eq:b''}
\begin{bmatrix}
b''(m) \\ F''(m)
\end{bmatrix}
=
X^{-1}(m)\left(\begin{bmatrix}
R_1'(m) \\
R_2'(m)
\end{bmatrix}
-X''(m)
\begin{bmatrix}
b(m) \\
F(m)
\end{bmatrix}\right).
\end{align}\end{small}
Hence, the free boundary $b\in C^{2}[0,\epsilon_2)$, and by definition we have $b\in C^2(\overline{m},\infty)$.
Then, we prove $b''(0)<0$ and thus $b''(m)<0$ on some interval $[0,\epsilon_3)$, implying the local concavity of $b$.
Letting $m=0$ in Eq. (\ref{eq:b''}) yields
\begin{small}$$b''(0)=\frac{1}{\det (X(0))}(I_1+I_2+I_3),$$\end{small}
where 
\begin{small}
\begin{align*}
I_1 &= q^2 \big[ (1\!-\!\rho)X_{12}(0) (\lambda_4^2 e^{-\lambda_3 b(0)} \!-\! \lambda_3^2 e^{-\lambda_4 b(0)}) - \rho X_{22}(0) (\lambda_2^2 e^{-\lambda_1 b(0)} \!-\! \lambda_1^2 e^{-\lambda_2 b(0)}) \big], \\[-0.5ex]
I_2 &= q \lambda_1 \lambda_2 X_{22}(0) \big[ \rho (\lambda_2 e^{-\lambda_1 b(0)} \!-\! \lambda_1 e^{-\lambda_2 b(0)}) + F(0) (e^{-\lambda_2 b(0)} \!-\! e^{-\lambda_1 b(0)}) \big], \\[-0.5ex]
I_3 &= q \lambda_3 \lambda_4 X_{12}(0) \big[ \!-(1\!-\!\rho) (\lambda_4 e^{-\lambda_3 b(0)} \!-\! \lambda_3 e^{-\lambda_4 b(0)}) + F(0) (e^{-\lambda_4 b(0)} \!-\! e^{-\lambda_3 b(0)}) \big].
\end{align*}
\end{small}
We deal with $I_1$, $I_2$ and $I_3$ respectively. According to Eqs.  (\ref{eq:spec33}) and (\ref{eq:spec34}), we have $I_1=0$. Besides, $I_2<0$ holds due to $\lambda_1<0$, $\lambda_2>0$, $b(0)>0$, $F(0)<0$ and $X_{22}(0)>0$. In order to prove $I_3<0$, we only need to prove $-(1-\rho)(\lambda_4e^{-\lambda_3b(0)}-\lambda_3e^{-\lambda_4b(0)})+F(0)(-e^{-\lambda_3b(0)}+e^{-\lambda_4b(0)})<0$. Using Eq. (\ref{eq:F(0)}), we have
\begin{small}
\begin{align*}
F(0)(e^{-\lambda_4b(0)} \!&-\! e^{-\lambda_3b(0)}) - (1\!-\!\rho)(\lambda_4e^{-\lambda_3b(0)} \!-\! \lambda_3e^{-\lambda_4b(0)}) \\[-0.5ex]
&\quad < [-(1\!-\!\rho)\lambda_4 \!-\! F(0)](e^{-\lambda_3b(0)} \!-\! e^{-\lambda_4b(0)}) < 0.
\end{align*}
\end{small}
As $\det(X(0))>0$, we obtain $b''(0)<0$. Hence, there exists $\epsilon_3\in(0,\overline{m}]$ such that $b''(m)<0$ holds for any $m\in [0,\epsilon_3)$. As $b'(0)=0$, we have $b'(m)<0$, and thus, $b$ is decreasing on $[0,\epsilon_3]$.

Finally, we prove that $b$ is decreasing on $[0,\infty)$. We only need to prove the monotonicity on $[0,\overline{m})$ as $b$ is a constant function on $[\overline{m},\infty)$. Define $\hat{m}=\inf \{m\in(0,\overline{m})|b'(m)=0\}\wedge  \overline{m}$. Assume $\hat{m}\ne \overline{m}$, then there exists $m_1>0$ satisfying $m_1=\inf \{m\in(0,\overline{m})|b'(m)=0\}$. We have $m_1>0$ because $b'(m)<0$ holds for any $m\in [0,\epsilon_3)$. As $b'(m_1)=0$, according to Eqs. (\ref{eq:ODEFb}) and (\ref{eq:G}), we obtain $G(b(m_1)-m_1)=0$, thus, $b(m_1)=b(0)+m_1>b(0)$ holds. This is a contradiction to the fact that $b'(m)\leq 0$ for any $m \in [0,m_1)$. Thus, we have $m_1=\overline{m}$, $b'(m)\leq 0$ holds for $m \in [0,\overline{m})$ and $b$ is decreasing on $[0,\infty)$. Therefore, we complete the proof.
\end{proof}
The results of Propositions \ref{prop:spec} and \ref{prop:b} indicate that the free boundary $b$ is decreasing on $[0,\infty)$. Economically, this implies that a poorer historical minimum performance level leads to a more conservative dividend payout policy due to the scarring effect. The local concavity of $b$ at $0$ suggests that its value stabilizes near $b(0)$ when the company has previously been close to bankruptcy. Furthermore, given an early-warning threshold $\overline{m}$, if the company's performance has never fallen below this level, the manager will optimally maintain a constant dividend-barrier strategy according to dividend signaling theory; see \citet*{AMG2013}. We will conduct numerical simulations to intuitively illustrate these observations in the next subsection.
\subsection{Numerical Simulation and Comparative Static Analysis}\label{subsec:num}
\begin{figure}[!ht]
    \centering
    \includegraphics[width=0.8\linewidth]{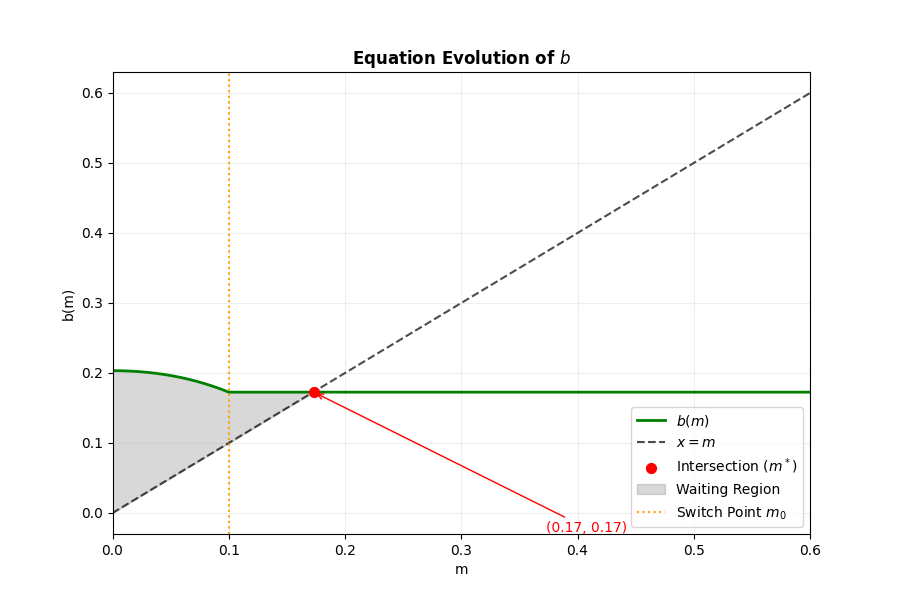}
    \caption{Numerical simulation of the equilibrium.}
    \label{fig:m0}
\end{figure}

The existence of the equilibrium is rigorously established in Section \ref{subsec:rigor}. Setting the parameters to $\mu=\sigma=1$, $\rho=0.3$, $q=5$, $\lambda_1=-2$, $\lambda_2=1$, $\lambda_3=-3$, $\lambda_4=2$, and $\overline{m}=0.1$, Figure \ref{fig:m0} illustrates the resulting waiting region $W$ (shaded area). For instance, with initial values $m=0.3$ and $x=0.5$, the company pays a lump-sum dividend at time zero, so that $X_0=\overline{m}=0.17$. Thereafter, the process $(X,M)$, as the unique strong solution of the associated Skorokhod reflection problem, stays within the waiting region $W$. For another example, if $m=0.05$ and $x=0.3$, the equilibrium singular control at $t=0$ immediately brings the wealth to the barrier, i.e., $X_0 = b(m)$.

When the historical minimum performance stays above the threshold $\overline{m}=0.1$, the firm follows a stable dividend policy, reflected in a constant barrier function $b$. Once the historical minimum falls below $\overline{m}$, however, the firm switches to a survival‑driven policy that reduces payouts. These results are consistent with dividend signaling theory,  dividend smoothing theory and the scarring effect; see, e.g. \citet*{AMG2013} and \citet*{K2021}.

\begin{figure}[!ht]
    \centering
    \includegraphics[width=1.0\linewidth]{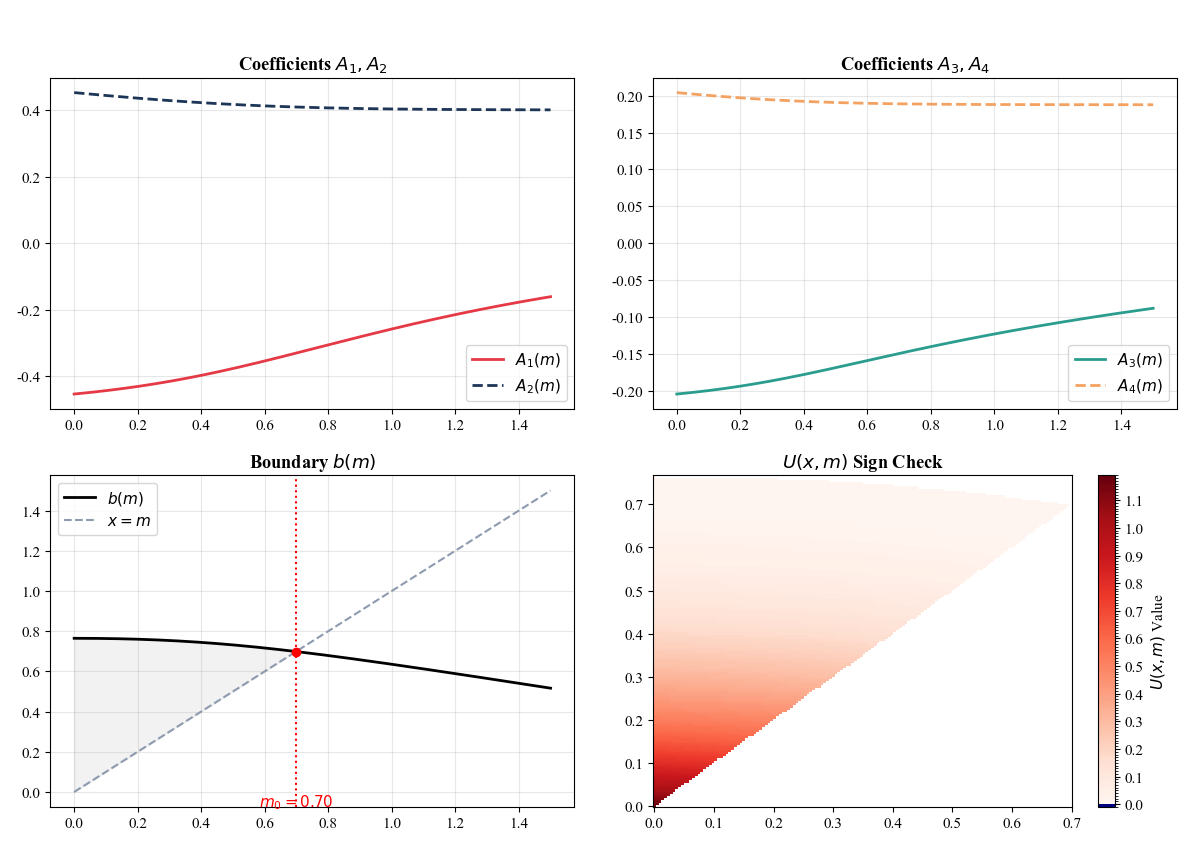}
    \caption{Numerical simulation of the equilibrium for $\overline{m}=\infty$.}
    \label{fig:ex1}
\end{figure}
\begin{figure}[!ht]
    \centering
    \includegraphics[width=1.0\linewidth]{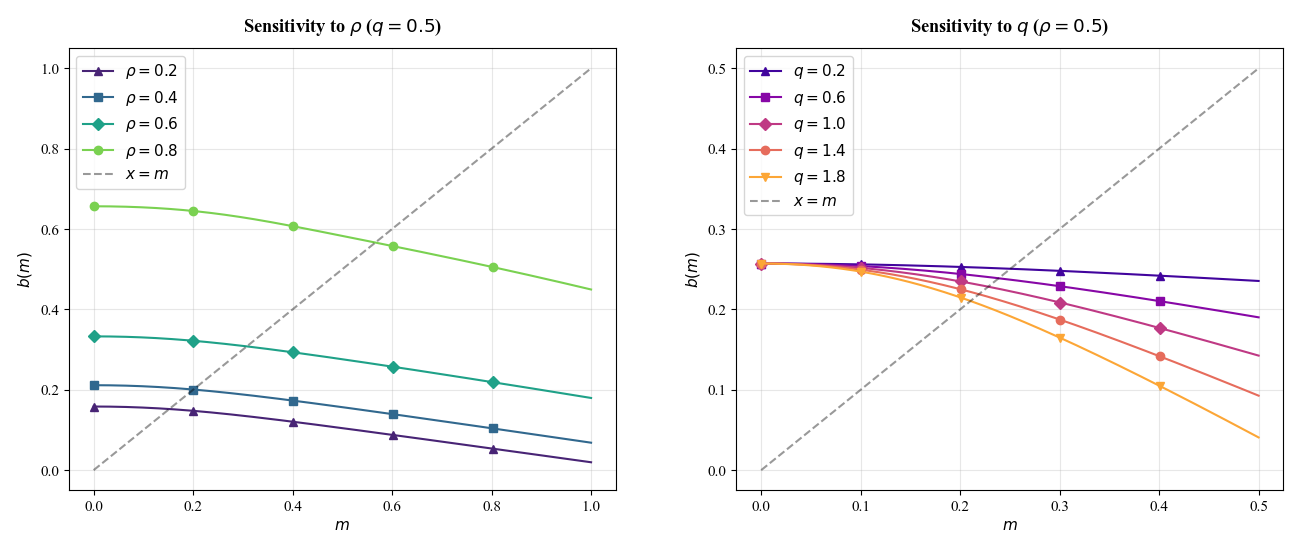}
    \caption{Comparative statics of the free boundary under different parameter values.}
    \label{fig:qrho}
\end{figure}
Set $\mu=\sigma=1$, $\rho=0.5$, $q=0.2$, $\lambda_1=-2.5$, $\lambda_2=0.5$, $\lambda_3=-3$, $\lambda_4=1$, and $\overline{m}=\infty$. Although this example does not satisfy the assumptions of Proposition 4.1 
, the numerical results suggest the existence of an equilibrium.  Figure \ref{fig:ex1} displays the results: the first two subplots show the functions $A_i$ ($i=1,2,3,4$), the third subplot presents the free boundary $b$, and the fourth subplot examines the validity of Eq. (4.4)
via a heat map. Here, the function $U$ (defined in the proof of Proposition 4.1
) is visualized; regions where $U\ge0$ are shaded in varying red, otherwise in uniform blue (no such region appears in this case). We focus on Eq. (4.4) 
because Eq. (4.3) 
has been verified directly. A sufficient condition for Eq. (4.4) 
is that $A_3(m)\le0$ for all $m\in[0,m^*]$, where $m^*$ is the first intersection of $b$ with the line $x=m$. As seen in the second subplot, this condition is met in the present example. For parameter values satisfying $|\lambda_i|<10\quad i=1,2,3,4$, $0<\rho<1$, and $0<q<5$, we perform extensive numerical simulations and find that the equilibrium always exists.

Let $\overline{m}=\infty$, $\mu=\sigma=1$, $\lambda_1=-2.5$, $\lambda_2=0.5$, $\lambda_3=-5$, and $\lambda_4=3$. Figure \ref{fig:qrho} contains two subplots, displaying the free boundary for different values of $\rho$ with $q$ fixed at $0.5$, and for different values of $q$ with $\rho$ fixed at $0.5$, respectively.

First, for fixed parameters and a given $m\ge0$, the boundary $b(m)$ is monotonically increasing in $\rho$. A larger $\rho$ makes the discount function closer to $e^{-\delta(t-s)}$ rather than $e^{-(\delta+\gamma)(t-s)}$. Hence, future managers become, on average, more patient. Anticipating this, the current manager expects that future selves will be less willing to pay dividends early. To reach a subgame-perfect equilibrium, the current manager therefore sets a higher dividend barrier.

Second, for fixed parameters and a given $m\geq0$, the boundary $b(m)$ is monotonically decreasing in $q$. From the form of the payoff functional, a larger $q$ implies that, for a fixed $m$, a dividend of the same size yields a higher marginal benefit when the historical minimum is lower. Consequently, when historical performance is poor, future managers will exhibit a stronger inclination to pay dividends. To counteract this strong propensity during difficult times, the current manager lowers the dividend barrier in the equilibrium strategy.

If $\rho=0$ or $\rho=1$, the framework reduces to a time-consistent singular control model that includes the running minimum process. If $q=0$, i.e., the running minimum is removed, the equilibrium becomes independent of $m$, and we obtain a constant barrier solution.
\section{Conclusion} 
We develop a path-dependent, time-inconsistent singular control framework incorporating a running minimum process. Constructing unique strong solutions for a class of Skorokhod reflection problems, we facilitate a rigorous equilibrium formulation. Compared to previous studies, our verification theorem operates under significantly weaker regularity and more precise integrability conditions.

Application to a dividend optimization problem reveals an equilibrium dividend barrier that decreases as the historical minimum deteriorates. This result provides a mathematical rationale for dividend smoothing theory, dividend signaling theory and the scarring effect observed in empirical payout policies. Our theoretical findings are further validated through numerical simulations.

Future research may address the general well-posedness of the Skorokhod reflection problems studied here, the uniqueness of equilibrium, models combining regular and singular controls, and other path-dependent settings, such as Asian-type path integrals that capture the cumulative influence of historical trajectories.

\bibliographystyle{plainnat}
\bibliography{bio}
\section{Some Auxiliary Proofs}
\subsection{Proof of Proposition \ref{prop:solution1}
}\label{app:1}
Let $(\Omega, \mathcal{F}, \{\mathcal{F}_r\}_{r \ge t}, P)$ be the filtered probability space. For a fixed $T > t$, we define the Banach space $\mathcal{S}^2[t, T]$ as the set of all adapted, c\'adl\'ag processes $X$ such that\begin{small}
\begin{align*}
\|X\|_{\mathcal{S}^2} = \left( E \left[ \sup_{t \le r \le T} |X_r|^2 \right] \right)^{\frac12} < \infty.
\end{align*}\end{small}
For any $X \in \mathcal{S}^2$, the running minimum process $\{M_r^X\}_{r\geq t}$ is defined as $
M_r^X = m \wedge \inf_{t \le s \le r} X_s.
$
As $|a\wedge b-a\wedge c|\leq |b-c|$ holds for any $a,b,c\in \R$, we have\begin{small}
\begin{align}\label{eq:MX}
\sup_{t \le s \le r} \Big|M_s^X - M_s^Y\Big|\leq \sup_{t \le s \le r} \sup_{t \le l \le s} \Big|(X_l - Y_l)\Big| = \sup_{t \le s \le r} |X_s - Y_s|.
\end{align}\end{small}

Given a singular $\mathrm{D} = \{\mathrm{D}_r\}_{r \ge t}$, we define the operator $\mathcal{L}$ on $\mathcal{S}^2$ such that for any $X \in \mathcal{S}^2$, the process $Z = \mathcal{L}(X)$ is given by\begin{small}
\begin{align*}
Z_r = x+y - \mathrm{D}_r  + \int_t^r \mu(X_s, M_s^X, \mathrm{D}_s, s) \d s + \int_t^r \sigma(X_s, M_s^X, \mathrm{D}_s, s) \d B_s.
\end{align*}\end{small}
The linear growth condition and the condition $\E_{x,m,y,t}[(\mathrm{D}_{t+r}-\mathrm{D}_{t-})^2]<\infty$ for any $r\geq 0$ guarantee that the above stochastic integral is well-defined.

Consider two processes $X, \hat{X} \in \mathcal{S}^2$ and let $Z = \mathcal{L}(X), \hat{Z} = \mathcal{L}(\hat{X})$. The difference is
\begin{small}
\begin{align*}
Z_r \!-\! \hat{Z}_r = & \int_t^r \!\! \big[ \mu(X_s, M_s^X\!, \mathrm{D}_s, s) \!-\! \mu(\hat{X}_s, M_s^{\hat{X}}\!, \mathrm{D}_s, s) \big] \d s \\[-0.5ex]
& + \int_t^r \!\! \big[ \sigma(X_s, M_s^X\!, \mathrm{D}_s, s) \!-\! \sigma(\hat{X}_s, M_s^{\hat{X}}\!, \mathrm{D}_s, s) \big] \d B_s.
\end{align*}
\end{small}
Using the Cauchy-Schwarz inequality and the BDG inequality yields\begin{small}
\begin{align*}
E \left[ \sup_{t \le s \le r} |Z_s - \hat{Z}_s|^2 \right] 
\leq & C \int_t^r E \left[ |X_s - \hat{X}_s|^2 + |M_s^X - M_s^{\hat{X}}|^2 \right] \d s,
\end{align*}\end{small}
where $C:=C(t,r,L)$. Thus, \begin{small}
\begin{align*}
E \left[ \sup_{t \le s \le r} |Z_s - \hat{Z}_s|^2 \right] \le 2C \int_t^r E \left[ \sup_{t \le u \le s} |X_u - \hat{X}_u|^2 \right] \d s.
\end{align*}\end{small}
This implies\begin{small} $$\| \mathcal{L}(X) - \mathcal{L}(\hat{X}) \|_{\mathcal{S}^2_r}^2 \le 2C \int_t^r \| X - \hat{X} \|_{\mathcal{S}^2_s}^2 \d s\leq C^*(t,r,L)\cdot \| X - \hat{X} \|_{\mathcal{S}^2_s}^2.$$\end{small} Picking $r>t$ such that $ C^*(t,r,L)<1$ holds. By the Banach fixed-point theorem, we obtain the existence and uniqueness of the SDE on $[t,r]$.

Because the choice of $C^*$ is independent of the initial distributions of $Z$, we similarly obtain the existence and uniqueness on $[r,2r-t]$, $[2r-t,3r-2t]$..., and thus the proof is completed.
\subsection{Proof of Proposition \ref{prop:Skorokhod}}\label{app:2}
(a) We employ the iterative method to construct the solution of the Skorokhod problem Eq. (\ref{eq:Skorokhod2})
. First, for any continuous and increasing function $l:[0,\infty)\to [0,\infty)$, we construct the mappings $\Phi_{l}$ and $\Gamma_{l}$. In the following, we abbreviate $\Phi_{l}$ and $\Gamma_{l}$ to $\Phi$ and $\Gamma$, respectively. Let $C[0, T]$ be the space of continuous functions equipped with the supremum norm $\|\cdot\|$. For any $w \in C[0, T]$, the running minimum path is given by $M(x)_t=m_0\wedge \inf\limits_{0\leq r\leq t} x_r$. For any $x\in C[0,T]$, the half-line Skorokhod problem
$z_t = x_t - k_t+y_0,\quad k_0=y_0$ on $(-\infty,0]$ has a unique solution $(z,k)$, where $k_t=y_0+\sup\limits_{s \le t} (x_s)^+$ according to Lemma 3.6.14 in \citet*{KS1991}. Define $\Phi:C[0,T]\to C[0,T]$ as $x\mapsto z$. For any $x\in C[0,T]$, consider the equation \begin{small}
\begin{align}\label{eq:Phiwz}
w_t&=\Phi(x_t-b(M(w)_t,l_t,t))+b(M(w)_t,l_t,t)\\&=x_t-\sup_{s\leq t}(x_s-b(M(w)_s,l_s,s))^+.\notag
\end{align}\end{small}
Using Eqs. (\ref{eq:MX}) and ({\ref{eq:Phiwz}}), for any $w,w'\in C[0,T]$, we have 
\begin{small}
\begin{align*}
&\|x_t \!-\! \sup_{s\leq t}(x_s \!-\! b(M(w)_s,l_s,s))^+ \!-\! x_t \!+\! \sup_{s\leq t}(x_s \!-\! b(M(w')_s,l_s,s))^+\| \\[-0.8ex]
&\leq \|\sup_{s\leq t}|(x_s \!-\! b(M(w')_s,l_s,s))^+ \!-\! (x_s \!-\! b(M(w)_s,l_s,s))^+|\| \\[-0.8ex]
&\leq \|b(M(w)_t,l_t,t) \!-\! b(M(w')_t,l_t,t)\| \\[-0.8ex]
&\leq L_m\|M(w)_t \!-\! M(w')_t\| \\[-0.8ex]
&\leq L_m\|w_t \!-\! w'_t\|,
\end{align*}
\end{small}
where $L_m\in(0,1)$ is the Lipschitz constant.
Thus, the map $w\mapsto \Phi(x-b(M(w),l,\cdot))+b(M(w),l,\cdot)$ is a contraction, and Eq. (\ref{eq:Phiwz}) has a unique solution $w$ in $C[0,T]$. Based on this result, giving $x$, we obtain a unique solution $(w,k)$ through Eq. (\ref{eq:Phiwz}), and we
define $\Gamma: C[0,T]\to C[0,T]$ as $x\mapsto w$. 

For any $x,x'\in C[0,T]$ and the corresponding $w,w'\in C[0,T]$, we have
\begin{small}\begin{align*}
\|w_t-w'_t\| &\leq \|x_t-\sup_{s\leq t}(x_s-b(M(w)_s,l_s,s))^+-x'_t+\sup_{s\leq t}(x'_s-b(M(w')_s,l_s,s))^+\|\\&\leq 2\|x_t-x'_t\|+L_m\|w_t-w'_t\| .
\end{align*}\end{small}
Therefore, letting $C_L:=\frac{2}{1-L_m}$,  we have
\begin{align}\label{eq:wx}
\|w_t-w'_t\|\leq C_L\|x_t-x'_t\|.
\end{align}
Then, we start the iteration as follows:\\
For any $\omega\in \Omega$, define $X_t^{(0)}(\omega)=x_0$ and $M_t^{(0)}(\omega)=m_0$, $\forall t\in[0,T]$. Moreover, for any adapted, c\'adl\'ag, continuous and increasing process $\{\mathrm{L}_t\}_{t\geq 0}$ satisfying $\E[\mathrm{L}_t^2]<\infty$ for any $t \geq 0$, define the sequence of iterations for 
$n\in \N$ as 
\begin{small}
\begin{align*}
X^{n+1}_t(\omega) &= \Gamma\bigg(X_0 \!+\! \int_0^t \!\! \mu(X^n_r(\omega),M_r^n(\omega),\mathrm{L}_r(\omega),r)\d r \!+\! \int_0^t \!\! \sigma(X^n_r(\omega),M_r^n(\omega),\mathrm{L}_r(\omega),r)\d B_r \bigg) \\[-0.6ex]
&:= \Gamma(I(X^n_r(\omega),M_r^n(\omega))),
\end{align*}
\end{small}
and \begin{small}$$M^{n+1}_t(\omega):=M(X^{n+1})_t(\omega)=m_0\wedge \inf_{0\leq r\leq t} X^{n+1}_r(\omega).$$\end{small} Then, usig Eq. (\ref{eq:wx}), we have\begin{small}
\begin{align*}
\|X^{n+2}_t(\omega)-X^{n+1}_t(\omega)\|+\|M^{n+2}_t(\omega)-M^{n+1}_t(\omega)\|&\leq 2\|X^{n+2}_t(\omega)-X^{n+1}_t(\omega)\|\\ &\leq 2C_L\|I(X^{n+1}_t(\omega))-I(X^{n}_t(\omega))\|.
\end{align*}\end{small}
Hence, using the BDG inequality and Lipschitz continuity of $\mu$ and $\sigma$, we have
\begin{small}
\begin{align*}
&\mathbb{E}\Big[ \sup_{s \le t}( |X^{(n+2)}_s \!-\! X^{(n+1)}_s|^2 \!+\! |M^{(n+2)}_s \!-\! M^{(n+1)}_s|^2) \Big] \\[-0.7ex]
&\le C_L^2 K \!\! \int_0^t \!\! \mathbb{E}\Big[ \sup_{r \le s} |X^{(n+1)}_r \!-\! X^{(n)}_r|^2 \!+\! \sup_{r \le s} |M^{(n+1)}_r \!-\! M^{(n)}_r|^2 \Big] \d s \\[-0.7ex]
&\le C_L^2 K t \cdot \mathbb{E}\Big[\sup_{s \le t}( |X^{(n+1)}_s \!-\! X^{(n)}_s|^2 \!+\! |M^{(n+1)}_s \!-\! M^{(n)}_s|^2)\Big] \\[-0.7ex]
&\le (C_L^2 K t)^n \mathbb{E}\Big[\sup_{s \le t}( |X^{(2)}_s \!-\! X^{(1)}_s|^2 \!+\! |M^{(2)}_s \!-\! M^{(1)}_s|^2)\Big].
\end{align*}
\end{small}
Define $T:=\frac{1}{2C_L^2K}$. using the Borel-Cantelli lemma and the Cauchy Convergence Principle, the sequence $\{X^{(n)}_t\}$ and $\{M^{(n)}_t\}$ converge uniformly on $\left[0,T\right]$, a.s.. According to Eq. (\ref{eq:Phiwz}) and the definition of $\Phi$, there exists a unique $\{\F_r\}_{r\geq t}$- adapted, nondecreasing, c\'adl\'ag, $k^{(n+1)}$ satisfying $k^{(n+1)}_0=y_0$ such that \begin{small}
\begin{align}\label{eq:XMK}
X^{(n+1)}_t-b(M^{(n+1)}_t,\mathrm{L}_t,t)&=x_0+y_0+\int_0^t\mu(X^{(n)}_r,M_r^{(n)},\mathrm{L}_r,r)\d r-b(M^{(n+1)}_t,\mathrm{L}_t,t)\\&\quad +\int_0^t\sigma(X^{(n)}_r,M_r^{(n)},\mathrm{L}_r,r)\d B_r-k^{(n+1)},\notag
\end{align}\end{small}
where $X^{(n+1)}_t-b(M^{(n+1)}_t,\mathrm{L}_t,t)\leq 0$ a.s., and the value of  $k_t^{(n+1)}$ only changes when $X^{(n+1)}_t-b(M^{(n+1)}_t,\mathrm{L}_t,t)=0$.
Because of the convergences of $\{X^{(n)}_t\}$ and $\{M^{(n)}_t\}$, we have $\{k^{(n)}_t\}$ converges uniformly on $[0,T]$.

Similarly adopting this approach within the intervals $[T,2T]$, $[2T,3T]$..., we obtain the a.s. locally uniform convergences of $\{X^{(n)}_t\}$, $\{M^{(n)}_t\}$ and $\{k^{(n)}_t\}$ on $[0,\infty)$. Denote by $X=\lim\limits_{n\to \infty}X^{(n)}$, $M=\lim\limits_{n\to \infty}M^{(n)}$, and $\mathrm{D}=\lim\limits_{n\to \infty}k^{(n)}$.

Passing $n\to \infty$ in Eq. (\ref{eq:XMK}) and using the Lipschitz continuity of $\mu$ and $\sigma$ with respect to the first and the second components, we obtain for a.s.-$\omega$,\begin{small}
\begin{align}\label{eq:SDEXM}
X_t=x_0+y_0+\int_0^t\mu(X_r,M_r,\mathrm{L}_r,r)\d r+\int_0^t\sigma(X_r,M_r,\mathrm{L}_r,r)\d B_r-\mathrm{D}_t,
\end{align}\end{small}
where $X_t-b(M_t,\mathrm{L}_t,t)\leq 0$ a.s. 

Then, we prove the uniqueness for the solution of the SDE (\ref{eq:SDEXM}). Assume that $(X^1,M^1)$ and $(X^2,M^2)$ are two different solutions of Eq. (\ref{eq:SDEXM})
. Then, we have \begin{small}
$$X^1_t = \Gamma(I(X^1, M^1))_t, \quad X^2_t = \Gamma(I(X^2, M^2))_t.$$\end{small}
Using the Lipschitz property of $\Gamma$ and imitating the proof of the existence, we obtain for $\Psi=X^2-X^1$ and $M^2-M^1$,\begin{small}
$$E\left[\sup_{0 \le s \le t} |\Psi^1_s - \Psi^2_s|^2\right] \le \tilde{K} \int_0^t E\left[\sup_{0 \le r \le s} |\Psi^1_r - \Psi^2_r|^2\right] \d s.$$\end{small}
According to the Gronwall inequality, for any $T>0$, we have $X^1_t=X^2_t$ and $M_t^1=M_t^2$ a.s. for any $t\in[0,T]$. Therefore, we obtain the uniqueness.

Finally, we prove the existence and uniqueness of the strong solution for the Skorokhod problem Eq. (\ref{eq:Skorokhod2}) 
. Let $\zeta:Q\to Q$, where $Q$ denotes the set of all the adapted, c\'adl\'ag, continuous and increasing process $\{\mathrm{L}_t\}_{t\geq 0}$ satisfying $\E[\mathrm{L}_t^2]<\infty$ for any $t \geq 0$. The mapping $\zeta$ is defined as 
\begin{align}\label{eq:zeta}
\zeta(\mathrm{L})_t= \mathrm{D}_t=y_0+\sup_{s\leq t}(I(X^{\mathrm{L}}_s,M^{\mathrm{L}}_s)-b(M_s,\mathrm{L}_s,s))^+,
\end{align}
where $\mathrm{D}_t$ is given by SDE (\ref{eq:SDEXM}). For $\mathrm{L}^1,\mathrm{L}^2\in Q$ and the corresponding $X^{\mathrm{L^1}}$, $X^{\mathrm{L^2}}$, $M^{\mathrm{L^1}}$, $M^{\mathrm{L^2}}$, $\mathrm{D}^1$ and $\mathrm{D}^2$, using Eq. (\ref{eq:SDEXM}), we have\begin{small}
\begin{align}\label{eq:XID}
\Big|X_t^{\mathrm{L}^1}-X_t^{\mathrm{L}^2}\Big|\leq \Big|I(X^{\mathrm{L}^1},M^{\mathrm{L}^1})-I(X^{\mathrm{L}^2},M^{\mathrm{L}^2})\Big|+|\mathrm{D}^2_t-\mathrm{D}^1_t|.
\end{align}\end{small}
Then, using Eqs. (\ref{eq:MX}), (\ref{eq:zeta}) and (\ref{eq:XID}), the BDG inequality and the Cauchy-Schwarz inequality yields \begin{small}
\begin{align*}
((1-L_m)^2-C_\epsilon t)E\left[\sup_{0\leq s\leq t}|X_s^{\mathrm{L}^1}-X_s^{\mathrm{L}^2}|^2\right]\leq (C_\epsilon t+\epsilon+L_d^2)E[|\mathrm{L}^1_t-\mathrm{L}^2_t|^2],
\end{align*}\end{small}
where $\epsilon:=\frac{(1-L_m)^2-L_d^2}{2}>0$, and $C_\epsilon$ is a constant related to $\epsilon$.
Hence, we have\begin{small}
\begin{align*}
E\left[\sup_{0\leq s\leq t}|X_s^{\mathrm{L}^1}-X_s^{\mathrm{L}^2}|^2\right]\leq C\cdot E[|\mathrm{L}^1_t-\mathrm{L}^2_t|^2]
\end{align*}\end{small}
for some sufficiently small $t>0$, where the constant $C\in (0,1)$.
Using Eq. (\ref{eq:SDEXM}) again, we have\begin{small}
\begin{align}\label{eq:DLK}
E\left[\sup_{0 \le s \le t} |\mathrm{D}^1_s-\mathrm{D}^2_s|^2\right] \le \tilde{K} \int_0^t E\left[\sup_{0 \le r \le s} |\mathrm{L}^1_r-\mathrm{L}^2_r|^2\right] \d s
\end{align}\end{small}
holds for a sufficiently small $t>0$, where $\tilde{K}$ is a constant.
Then, similarly, using the iterative method, we obtain $\mathrm{D}\in Q$ such that\begin{small}
\begin{align*}
X_t=x_0+y_0+\int_0^t\mu(X_r^{\mathrm{D}},M_r^{\mathrm{D}},\mathrm{D}_r,r)\d r+\int_0^t\sigma(X_r^{\mathrm{D}},M_r^{\mathrm{D}},\mathrm{D}_r,r)\d B_r-\mathrm{D}_t,
\end{align*}\end{small}
where we have
$\mathrm{D}_t=y_0+\sup\limits_{s\leq t}(I(X^{\mathrm{D}}_s,M^{\mathrm{D}}_s)-b(M_s,\mathrm{D}_s,s))^+$ and $\mathrm{D}$ is an increasing process. The Lipschitz continuity of $\mu$ and $\sigma$, along with Eq. (\ref{eq:DLK}) ensures the uniqueness of $\mathrm{D}$ by an argument analogous to that used in the preceding proof of the uniqueness of $X$ and $M$. Therefore, we obtain the existence and the uniqueness of the solution for the Skorohod problem Eq. (\ref{eq:Skorokhod2}) 
.

(b) We prove the admissibility of $\Xi$. Denote by $J$ the value of the integral in Condition (b) of Definition \ref{def:sing2}
. There exists $C>0$ such that
\begin{small}
\begin{align*}
J &\leq C \Bigg(\E_{x,m,y,t}\Bigg[\int_0^{\tau_a^\mathrm{D}} \!\! \beta(r-s)\d r\Bigg] + \E_{x,m,y,t}\Bigg[\int_0^{\tau_a^\mathrm{D}} \!\! \beta(r-s)\diamond\d \mathrm{D}_r\Bigg] \\[-0.6ex]
&\quad + \E_{x,m,y,t}\Bigg[\int_0^{\tau_a^\mathrm{D}} \!\! \beta(r-s)\square \d M_r\Bigg]\Bigg) := C(J_1+J_2+J_3).
\end{align*}
\end{small}
We obviously have\begin{small}
\begin{align*}
J_1&:=\E_{x,m,y,t}\Bigg[\int_0^{\tau_a^\mathrm{D}}\beta(r-s)\d r\Bigg]\leq C\sum_{n=1}^\infty  n^{-C_2}<\infty.
\end{align*}\end{small}
Because $\E\left[\sup_{0\leq r\leq 1}\left|B_r-B_0\right|\right]=\sqrt{\frac{2}{\pi}}$, using Eq. (\ref{eq:Phiwz}) yields
\begin{small}
\begin{align*}
|\mathrm{D}_{n+1}\!-\!\mathrm{D}_n| &\le \sup_{r,s\in[t,t+1]}\!(|I(X_r,M_r)\!-\!I(X_s,M_s)|)\!+\!\sup_{r,s\in[t,t+1]}\!(b(M_r,\mathrm{D}_r,r)\!-\!b(M_s,\mathrm{D}_r,s))^+ \\[-0.7ex]
&\le C\Big[1\!+\!\sup_{r,s\in[t,t+1]}\!|I(X_r,M_r)\!-\!I(X_s,M_s)|\Big]\!+\!L_m\sup_{r,s\in[t,t+1]}\!|X_r\!-\!X_s|\!+\!L_d |\mathrm{D}_{n+1}\!-\!\mathrm{D}_n| \\[-0.7ex]
&\le C\Big(1\!+\!\sup_{r\in[0,1]}|B_r|\Big) < \infty.
\end{align*}
\end{small}
Thus, we obtain
\begin{small}
\begin{align*}
J_2 &:= \E_{x,m,t}\Bigg[\int_0^{\tau_a^\mathrm{D}} \!\! \beta(r-s)\diamond\d \mathrm{D}_r\Bigg] = \sum_{n=1}^{\infty}\E_{x,m,t}\Bigg[\int_{(n-1)\wedge \tau_a^{\mathrm{D}}}^{n\wedge \tau_a^\mathrm{D}} \!\! \beta(r-s)\diamond\d \mathrm{D}_r\Bigg] \\[-0.8ex]
&\leq C \sum_{n=1}^{\infty} n^{-C_2} \E_{x,m,t}\bigg[\int_{(n-1)\wedge \tau_a^{\mathrm{D}}}^{n\wedge \tau_a^\mathrm{D}} \!\! 1\diamond \d \mathrm{D}_r\bigg] \\[-0.8ex]
&\leq C \sum_{n=1}^\infty n^{-C_2} \Big(1 + \E\big[\sup_{r\in [n,n+1)} |B_r|\big]\Big) \\[-0.8ex]
&\leq C \sum_{n=1}^\infty n^{-C_2} < \infty.
\end{align*}
\end{small}
Similarly, we have $J_3:=\E_{x,m,t}\Bigg[\int_0^{\tau_a^\mathrm{D}}\beta(r-s)\square \d M_r\Bigg]<\infty$.
Therefore, $J<\infty$ and the proof is complete.
\subsection{Proof of Proposition \ref{prop:f^s} 
}\label{app:3}
Throughout this proof, $K$ represents a positive constant, not necessarily the same at each occurrence. As $h$ can be chosen independently of $K$, we always assume $2Kh<1$. For notational simplicity, we suppress the stopping time $\tau_a^{\mathrm{D}^h}$ in the subsequent analysis. Adopting the same notations as in the proof of Theorem \ref{th:verify}.

Observe that $\sup\limits_{t\leq r\leq t+h}|M_r^{\mathrm{D}^h}|\leq |A|+|m|$, where $A$ denotes a lower bound of the function $a$.
Using Eq. (\ref{eq:X0}) 
and Condition (a) in Assumption \ref{as:allpaper} 
yields
\begin{small}
\begin{align*}
\mathrm{D}_{t+h}^h-\mathrm{D}_{t-}^{h} &= X^{\mathrm{D}^h}_{t-}-X_{t+h}^{\mathrm{D}^h} + \int_t^{t+h}\!\!\mu(X_r^{\mathrm{D}^h},M_r^{\mathrm{D}^h},\mathrm{D}_r,r)\d r + \int_t^{t+h}\!\!\sigma(X_r^{\mathrm{D}^h},M_r^{\mathrm{D}^h},\mathrm{D}_r,r)\d B_r \\[-0.8ex]
&\leq K + \sup_{r\in[t,t+h]}\{X_r^{\mathrm{D}^h}\} + hK\big(1+\sup_{r\in[t,t+h]}\{X_r^{\mathrm{D}^h}\} + \mathrm{D}_{t+h}^h-\mathrm{D}_{t-}^{h}\big) + I(h).
\end{align*}
\end{small}
for any perturbation $\{\mathrm{D}^{h}_r\}_{r\geq t}\in \mathcal{D}_t$, where we define
$$I(h):=\int_t^{t+h}\sigma(X_r^{\mathrm{D}^h},M_r^{\mathrm{D}^h},\mathrm{D}_r,r)\d B_r.$$
Thus, applying the Cauchy-Schwarz inequality and the It\^{o} isometry yields 
\begin{small}
\begin{align*}
\mathbb{E}\big[(\mathrm{D}_{t+h}^h)^2\big] &\leq K + K\mathbb{E}\big[(\mathrm{D}_{t+h}^h-\mathrm{D}_{t-}^{h})^2\big] \\[-0.8ex]
&\leq K \mathbb{E}\Big[\big(1 + \sup_{r\in[t,t+h]}\{X_r^{\mathrm{D}^h}\}\big)^2\Big] + K\mathbb{E}[I^2(h)] \\[-0.8ex]
&\leq K \mathbb{E}\Big[\big(1 + \sup_{r\in[t,t+h]}\{X_r^{\mathrm{D}^h}\}\big)^2\Big] + hK\mathbb{E}\Big[1 + \sup_{r\in[t,t+h]}\{(X_r^{\mathrm{D}^h})^2\} + (\mathrm{D}_{t+h}^h)^2\Big].
\end{align*}
\end{small}
Therefore, we have
\begin{small}
\begin{align}\label{eq:D-D}
\E\big[\left(\mathrm{D}_{t+h}^h\right)^2\big]\leq K \E\Big[\Big(1+\sup_{t\leq r\leq t+h}\{X_r^{\mathrm{D}^h}\}\Big)^2\Big].
\end{align}
\end{small}
Similarly, 
\begin{small}
\begin{align}\label{eq:DX}
&\E\Big[\big(1+\sup_{r\in[t,t+h]}\{X_r^{\mathrm{D}^h}\}\big)(\mathrm{D}_{t+h-}^h-\mathrm{D}_{t}^{h})\Big] \\[-0.5ex]
&\leq K \E\Big[\big(1+\sup_{r\in[t,t+h]}\{X_r^{\mathrm{D}^h}\}\big)^2\Big] + K\sqrt{\E[I^2(h)] \cdot \E\Big[\big(1+\sup_{r\in[t,t+h]}\{X_r^{\mathrm{D}^h}\}\big)^2\Big]}\notag \\[-0.5ex]
&\leq K \E\Big[\big(1+\sup_{r\in[t,t+h]}\{X_r^{\mathrm{D}^h}\}\big)^2\Big] + Kh\E[(\mathrm{D}^h_{t+h})^2]\notag \\[-0.5ex]
&\leq K \E\Big[\big(1+\sup_{r\in[t,t+h]}\{X_r^{\mathrm{D}^h}\}\big)^2\Big].\notag
\end{align}
\end{small}
where we have just used Eq. (\ref{eq:D-D}) in the last inequality.
Moreover, for any $s\in[t,t+h]$, we have
$\int_t^{s}(X_r^{\mathrm{D}^h}-A)\,\mathrm{d}\mathrm{D}_r^h\geq 0$.

Next,  we show
\begin{small}
$$\E_{x,m,y,t}\!\Big[\bigl(\sup\limits _{s\in[t,t+h]}|X_s^{\mathrm{D}^h}|\bigr)^2\Big]<\infty.$$
\end{small}
Applying the It\^{o}-Tanaka-Meyer formula, we obtain
\begin{small}
\begin{align}\label{eq:fx1}
&\sup_{s\in[t,t+h]}\!(X^{\mathrm{D}^h}_{s}\!-\!A)^2 \\[-0.5ex]
&\le \sup_{s\in[t,t+h]} \bigg\{ \int_t^{s} \!\! \Big[ 2(X_r^{\mathrm{D}^h}\!-\!A)\mu + \sigma^2 \Big] (X_r^{\mathrm{D}^h},M_r^{\mathrm{D}^h},\mathrm{D}^h_r,r) \mathrm{d}r\notag \\[-0.5ex]
&\qquad - \int_t^{s} \! 2(X_r^{\mathrm{D}^h}\!-\!A)\,\mathrm{d}\mathrm{D}_r^h + \int_t^{s} \! 2(X_r^{\mathrm{D}^h}\!-\!A)\sigma(X_r^{\mathrm{D}^h},M_r^{\mathrm{D}^h},\mathrm{D}^h_r,r)\,\mathrm{d}B_r \bigg\} + K \notag\\[-0.5ex]
&\le K + Kh \sup_{s\in[t,t+h]}\!(X^{\mathrm{D}^h}_{s})^2 + \sup_{s\in[t,t+h]} \int_t^{s} \! 2(X_r^{\mathrm{D}^h}\!-\!A)\sigma(X_r^{\mathrm{D}^h},M_r^{\mathrm{D}^h},\mathrm{D}^h_r,r)\,\mathrm{d}B_r.\notag
\end{align}
\end{small}
Using the Burkholder-Davis-Gundy (abbr. BDG) inequality and the Cauchy-Schwarz inequality yields
\begin{small}
\begin{align}\label{eq::fx2}
&\mathbb{E}\bigg[\sup_{s\in[t,t+h]} \int_t^{s} 2(X_r^{\mathrm{D}^h}-A)\sigma(X_r^{\mathrm{D}^h},M_r^{\mathrm{D}^h},\mathrm{D}^h_r,r)\d B_r \bigg]\notag \\[-0.5ex]
&\le K\mathbb{E}\bigg[ \Big( \int_t^{t+h} (X_s^{\mathrm{D}^h}-A)^2 (1+|X_s^{\mathrm{D}^h}|)^2 \d s \Big)^{1/2} \bigg]\notag \\[-0.5ex]
&\le K\mathbb{E}\bigg[ \big(1+\sup_{s\in[t,t+h]}|X_s^{\mathrm{D}^h}|\big) \Big( \int_t^{t+h} (1+|X_s^{\mathrm{D}^h}|)^2 \d s \Big)^{1/2} \bigg]\notag \\[-0.5ex]
&\le \mathbb{E}\bigg[ \frac{1}{4} \big(\sup_{s\in[t,t+h]}|X_s^{\mathrm{D}^h}|\big)^2 + K \int_t^{t+h} (X_s^{\mathrm{D}^h})^2 \d s + K \bigg] \notag\\[-0.5ex]
&\le \mathbb{E}\bigg[ \big(\frac{1}{4} + Kh\big) \sup_{s\in[t,t+h]} |X_s^{\mathrm{D}^h}|^2 + K \bigg].
\end{align}
\end{small}
In addition, we have
\begin{small}
\begin{align}\label{eq:fx3}
\sup_{s\in[t,t+h]}\!(X_s^{\mathrm{D}^h}-A)^2
&\ge \sup_{s\in[t,t+h]}\!|X_s^{\mathrm{D}^h}|^2 - 2|A|\!\sup_{s\in[t,t+h]}\!|X_s^{\mathrm{D}^h}| + A^2 \notag \\[-0.8ex]
&\ge \frac{3}{4} \sup_{s\in[t,t+h]}\!|X_s^{\mathrm{D}^h}|^2 - 3A^2.
\end{align}
\end{small}
Combining Eqs. (\ref{eq:fx1})-(\ref{eq:fx3}) yields
\begin{small}
\begin{align}\label{eq:X^2<inf}
\E\!\Big[\sup_{s\in[t,t+h]}|X^{\mathrm{D}^h}_{s}|^2\Big]<\infty
\end{align}
\end{small}
for some sufficiently small $h>0$.

Following the proof of Theorem \ref{th:verify}
, 
in order to verify the equilibrium property, it is sufficient to show
\begin{small}
\begin{align}
\limsup_{h\to0+}\E\Big[\int_t^{t+h} \!\! \frac{f_s(X_{(t+h)-}^{\mathrm{D}^h},M_{(t+h)-}^{\mathrm{D}^h},\mathrm{D}^h_{(t+h)-},t+h,r) -f_s(X_r^{\mathrm{D}^h},M_r^{\mathrm{D}^h},\mathrm{D}^h_r,r,r)}{h}  \d r \Big] \le 0, \label{eq:J11} \\[-0.5ex]
\limsup_{h\to0+}\E\Big[\int_t^{t+h} \!\! \tfrac{\beta(r-t)-1}{h} H(X_r^{\mathrm{D}^h},M_r^{\mathrm{D}^h},\mathrm{D}^h_r,r) \d r \Big] \le 0, \label{eq:J12} \\[-0.5ex]
\limsup_{h\to0+}\E\Big[\int_t^{t+h} \!\! \tfrac{\beta(r-t)-1}{h} c(X_r^{\mathrm{D}^h},M_r^{\mathrm{D}^h},\mathrm{D}^h_r,r) \diamond \d \mathrm{D}^h_r \Big] \le 0, \label{eq:J13} \\[-0.5ex]
\limsup_{h\to0+}\E\Big[\int_t^{t+h} \!\! \tfrac{\beta(r-t)-1}{h} h(M_r^{\mathrm{D}^h},\mathrm{D}^h_r,r) \square \d M_r^{\mathrm{D}^h} \Big] \le 0. \label{eq:J14}
\end{align}
\end{small}
In fact, adding Eqs. (\ref{eq:J11})-(\ref{eq:J14}) to Eq. (\ref{eq:JJ})
and using Eq. (\ref{eq:ver1})
yield Eq. (\ref{eq:equilibrium})
,  thus we conclude that $\hat{\mathrm{D}}$ is an equilibrium singular control.

We now prove Eqs. (\ref{eq:J11}) and (\ref{eq:J13}); the proofs of Eqs. (\ref{eq:J12}) and (\ref{eq:J14}) are analogous.  
Using Eq. (\ref{eq:module})
, the Jensen inequality, Eqs. (\ref{eq:D-D}) and (\ref{eq:X^2<inf}), we obtain
\begin{small}
\begin{align*}
&\mathbb{E}\bigg[\int_t^{t+h} \!\! \Big| \tfrac{f_s(X_{(t+h)-}^{\mathrm{D}^h},M_{(t+h)-}^{\mathrm{D}^h},\mathrm{D}^h_{(t+h)-},t+h,r) - f_s(X_r^{\mathrm{D}^h},M_r^{\mathrm{D}^h},\mathrm{D}^h_r,r,r)}{h} \Big| \d r \bigg] \\[-0.8ex]
&\le K\mathbb{E}\Big[\sup_{r\in[t,t+h]}\!\big\{1+\sup_{s\in[t,t+h)}\!(X^{\mathrm{D}^h}_{s})^2+\sup_{s\in[t,t+h)}\!(M^{\mathrm{D}^h}_{s})^2+(\mathrm{D}_{t+h}^h)^2\big\}\Big] < \infty.
\end{align*}
\end{small}
Thus, the integrability condition for Fatou's lemma holds. Then, by the right-continuity of $\{X^{\mathrm{D}^h}_s\}_{s\geq t}$ and $\{\mathrm{D}^h_s\}_{s\geq t}$, we have
\begin{small}
\begin{align*}
\limsup_{h\to 0^+} \mathbb{E} \bigg[ \frac{1}{h} \int_t^{t+h} \!\!\! \big( f_s(\textstyle X_{(t+h)\text{-}}^{\mathrm{D}^h},M_{(t+h)\text{-}}^{\mathrm{D}^h},\mathrm{D}^h_{(t+h)\text{-}},t\!+\!h,r) - f_s(X_{r}^{\mathrm{D}^h},M_{r}^{\mathrm{D}^h},\mathrm{D}^h_{r},r,r) \big) \mathrm{d}r \bigg] \le 0.
\end{align*}
\end{small}
Then we prove Eq. (\ref{eq:J13}). Using Eqs. (\ref{eq:module2})
, (\ref{eq:D-D}) and (\ref{eq:DX}), we obtain
\begin{small}
\begin{align}
&\E\bigg[\int_t^{t+h} \!\! \Big| \tfrac{\beta(r-t)-1}{h} c(X_r^{\mathrm{D}^h},M_r^{\mathrm{D}^h},\mathrm{D}_r^h,r) \Big| \diamond \d\mathrm{D}_r^h \bigg]  \\[-0.6ex]
&\le K\E\bigg[ \int_t^{t+h} \!\! \tfrac{r-t}{h} \big(1+\sup_{s\in[t,t+h)}\!|X_s^{\mathrm{D}^h}|+\mathrm{D}_{t+h}^h\big) \diamond \d\mathrm{D}_r^h \bigg] \notag \\[-0.6ex]
&\le K\E\bigg[ \big(1+\sup_{s\in[t,t+h)}\!|X_s^{\mathrm{D}^h}|+\mathrm{D}_{t+h}^h\big) \!\! \int_t^{t+h} \!\! \diamond \d\mathrm{D}_r^h \bigg] \le K\E\Big[1+\big(\sup_{s\in[t,t+h)}\!|X_s^{\mathrm{D}^h}|+\mathrm{D}_{t+h}^h\big)^2\Big]\notag\\& < \infty,\notag
\end{align}
\end{small}
Fix an $\omega\in\Omega$ and take $h<1$. By the right‑continuity of $\{\mathrm{D}_t^h\}_{t\ge0}$ and  Fatou's lemma, we obtain
\begin{small}
\begin{align*}
&\limsup_{h\to0+}\E\bigg[\int_t^{t+h} \!\! \tfrac{\beta(r-t)-1}{h} c(X_r^{\mathrm{D}^h},M_r^{\mathrm{D}^h},\mathrm{D}_r^h,r) \diamond \d\mathrm{D}_r^h(\omega) \bigg] \\[-0.6ex]
&\le K\E\bigg[\limsup_{h\to0+} \Big\{ \big(1+\sup_{s\in[t,t+1)}\!|X_s^{\mathrm{D}^h}(\omega)|\big) \!\! \int_t^{t+h} \!\! \diamond \d\mathrm{D}_r^h(\omega) \Big\} \bigg] \\[-0.6ex]
&= K\E\bigg[\limsup_{h\to0+} \Big\{ \big(1+\sup_{s\in[t,t+1)}\!|X_s^{\mathrm{D}^h}(\omega)|\big) \big(\mathrm{D}_{t+h}(\omega)-\mathrm{D}_t(\omega)\big) \Big\} \bigg] = 0.
\end{align*}
\end{small}
Thus, the proof is complete.
\end{document}